\documentclass[aap]{imsart}

\RequirePackage{amsthm,amsmath,amsfonts,amssymb,amsbsy,latexsym,dsfont,bbm}
\RequirePackage[numbers]{natbib}
\RequirePackage[colorlinks,citecolor=blue,urlcolor=blue]{hyperref}

\usepackage{enumerate}

\startlocaldefs
\newenvironment{enumeratei}{\begin{enumerate}[\upshape (i)]}{\end{enumerate}}
\newenvironment{enumeratea}{\begin{enumerate}[\upshape (a)]}{\end{enumerate}}

\usepackage{paralist}

\newenvironment{inparaenuma}{\begin{inparaenum}[\upshape \bfseries (a) ]}{\end{inparaenum}}

\numberwithin{equation}{section}
\numberwithin{figure}{section}
\numberwithin{table}{section}

\newtheorem{thm}{Theorem}[section]
\newtheorem{lem}[thm]{Lemma}
\newtheorem{theorem}[thm]{Theorem}

\newtheorem{corollary}[thm]{Corollary}
\newtheorem{prop}[thm]{Proposition}
\newtheorem{defn}[thm]{Definition}

\newtheorem{ass}[thm]{Assumption}

\newtheorem{lemma}[thm]{Lemma}

\theoremstyle{definition}
\newtheorem{rem}{Remark}

\newcommand{\ind}{\mathds{1}}
\newcommand{\eps}{\varepsilon}

\newcommand{\set}[1]{\left\{#1\right\}}

\newcommand{\convas}{\stackrel{\mathrm{a.s.}}{\longrightarrow}}

\newcommand{\stod}{\preceq_{\mathrm{st}}}

\def\qed{ \hfill $\blacksquare$}

\newcommand{\cA}{\mathcal{A}}\newcommand{\cB}{\mathcal{B}}
\newcommand{\cD}{\mathcal{D}}\newcommand{\cF}{\mathcal{F}}

\newcommand{\cS}{\mathcal{S}}\newcommand{\cT}{\mathcal{T}}

\newcommand{\cY}{\mathcal{Y}}

\newcommand{\vt}{\mathbf{t}}

\newcommand{\bL}{\mathbb{L}}
\newcommand{\bN}{\mathbb{N}}
\newcommand{\bR}{\mathbb{R}}
\newcommand{\bT}{\mathbb{T}}

\newcommand{\bZ}{\mathbb{Z}}

\DeclareMathOperator{\E}{\mathbb{E}}
\DeclareMathOperator{\pr}{\mathbb{P}}

 \DeclareMathOperator{\BP}{BP}

\newcommand{\sss}{\scriptscriptstyle}

\newcommand{\erdos}{Erd\H{o}s-R\'enyi }

\newcommand{\convd}{\stackrel{d}{\longrightarrow}}
\newcommand{\convp}{\stackrel{P}{\longrightarrow}}

\newcommand{\yu}{{\sf Yule} }

\newcommand{\LPA}{{\sf LPA} }

\newcommand{\desc}[2]{{#1}_{#2 \downarrow}}

\usepackage{accents}

\newcommand{\LBP}{\underline{\BP}}
\newcommand{\LZF}{\underline{Z}_f}
\newcommand{\LW}{\underline{W}_\infty}

\endlocaldefs

\begin{document}

\begin{frontmatter}
\title{Root finding algorithms and persistence of Jordan centrality in growing random trees}
\runtitle{Persistence of Jordan centrality}

\begin{aug}
\author[A]{\fnms{Sayan} \snm{Banerjee}\ead[label=e1]{sayan@email.unc.edu}}
\and
\author[A]{\fnms{Shankar} \snm{Bhamidi}\ead[label=e2,mark]{bhamidi@email.unc.edu}}
\address[A]{Department of Statistics and Operations Research, 304 Hanes Hall, University of North Carolina, Chapel Hill, NC 27599,
\printead{e1,e2}}
\end{aug}

\begin{abstract}
We consider models of growing random trees $\set{\cT_f(n):n\geq 1}$ with model dynamics driven by an attachment function $f:\bZ_+\to \bR_+$. At each stage a new vertex enters the system and connects to a vertex $v$ in the current tree with probability proportional to $f(\text{degree}(v))$. The main goal of this study is to understand the performance of root finding algorithms. A large body of work (e.g. \cite{bubeck2017finding,jog2016analysis,jog2018persistence}) has emerged in the last few years in using techniques based on the Jordan centrality measure \cite{jordan1869assemblages} and its variants to develop root finding algorithms. Given an unlabelled unrooted tree, one computes the Jordan centrality for each vertex in the tree and for a fixed budget $K$ outputs the optimal $K$ vertices (as measured by Jordan centrality). Under general conditions on the attachment function $f$, we derive necessary and sufficient bounds on the budget $K(\eps)$ in order to recover the root with probability at least $1-\eps$. For canonical examples such as linear preferential attachment and uniform attachment, these general results give matching upper and lower bounds for the budget. We also prove persistence of the optimal $K$ Jordan centers for any $K$, i.e. the existence of an almost surely  finite random time $n^*$ such that for $n \geq n^*$ the identity of the $K$-optimal Jordan centers in $\set{\cT_f(n):n\geq n^*}$ does not change,  thus describing robustness properties of this measure. Key technical ingredients in the proofs of independent interest include sufficient conditions for the existence of exponential moments for limits of (appropriately normalized) continuous time branching processes within which the models $\set{\cT_f(n):n\geq 1}$ can be embedded,  as well as rates of convergence results to these limits.
\end{abstract}



\end{frontmatter}

\section{Introduction}
\label{sec:int}

\subsection{Motivation}
Driven by the explosion in the amount of data on various real world networks, the last few years have seen the emergence of many new mathematical network models. Goals underlying these models include \begin{inparaenuma}
	\item extracting unexpected connectivity patterns within in the network (e.g. community detection);
	\item understand properties of dynamics on these real world systems such as the spread of epidemics and opinion dynamics;
	\item understanding mechanistic reasons for the emergence of empirically observed properties of these systems such as heavy tailed degree distribution or the small world property. 
\end{inparaenuma} 
We refer the interested reader to \cite{albert2002statistical,newman2003structure,newman2010networks,bollobas2001random,durrett-rg-book,van2009random} and the references therein for a starting point to the vast literature on network models.

 An increasingly important area in network analysis is the setting of dynamic or temporal networks (see \cite{holme2012temporal,masuda2016guidance} and the references therein): networks that evolve and change over time. In the context of the probability community, one question that has attracted increasing attention falls under the branch of \emph{network archaeology} \cite{navlakha2011network}: trying to understand the evolution of a network based on only observing the current structure of the network. Prototypical questions in this area include the so-called \emph{Adam problem} \cite{bubeck2017finding,bubeck-mossel}:  can one reconstruct the original seed that was the genesis of the current network based off of only current structural (topological) information of the network without any other information such as time-stamps, labels etc.   These questions have implications in a number of fields ranging from systems biology \cite{navlakha2011network,young2018network} to the detection of sources of malicious information on social networks \cite{shah2012rumor}. We give a precise description of the class of evolving network models considered in this paper in the next section and then describe the goals of seed-finding algorithms.

\subsection{Model definition}
\label{sec:model}
Fix an attachment function $f$.
 We will grow a sequence of random trees $\set{\cT_{f}(j):2\leq j\leq n }$ as follows (see Section \ref{sec:disc} regarding related work for more general network models): 

\begin{enumeratei}
	\item For $n=2$, $\cT_{f}(j)$ consists of two vertices attached by a single edge. Label these as $\{v_1, v_2\}$ and call the vertex $v_1$ as the ``root'' of the tree. 
	\item Fix $n> 2$. Let the vertices in $\cT_{f}({n-1})$ be labeled as $\{v_1,\dots,v_{n-1}\}$. For each vertex $v\in \cT_{f}({n-1})$  let $\text{deg}(v)$ denote the degree of $v$.  A new vertex $v_n$ enters the system. Conditional on $\cT_{f}({n-1})$, this new vertex attaches to a currently existing vertex $v\in \cT_{f}({n-1})$ with probability proportional to $f(\text{deg}(v))$. Call the vertex that $v_n$ attaches to, the ``parent'' of $v_n$ and direct the edge from this parent to $v_n$ resulting in the tree $\cT_{f}(n)$. 
\end{enumeratei}

Thus, for any $n$, $\cT_{f}(n)$ is a random tree on $n$ labelled vertices $\{v_1,\dots,v_n\}$ rooted at $v_1$. To fix ideas, explicit examples of attachment functions include:

\begin{enumeratea}
	\item {\bf Uniform attachment:} $f(\cdot)\equiv 1$. In this case new vertices attach to existing vertices uniformly at random. The corresponding random tree is called the random recursive tree and has been heavily analyzed across probabilistic combinatorics and computer science \cite{smythe1995survey,drmota2009random}. 
	\item {\bf ``Pure'' preferential attachment:} $f(k) = k$. Here new vertices attach to pre-existing vertices with probability proportional purely to the degree of the existing vertex \cite{barabasi1999emergence}. 
	\item {\bf Affine Preferential attachment:} $f(k) = k+\beta$ for $\beta \ge 0$ \cite{durrett2007random}. This is a more general analogue of the pure preferential attachment model above in which one can `fine tune' the attachment probability of the lower degree vertices via the parameter $\beta$.
	\item {\bf Sublinear attachment:}  We define sublinear attachment as in \cite[Def. 2]{jog2016analysis}: An attachment scheme with function $f$ is called sublinear if:
	\begin{enumeratei}
		\item $f$ is non-decreasing and not identically equal to $1$ with $f(k)\geq 1$ for all $k \ge 1$. 
		\item There exists $\alpha \in (0,1)$ such that $f(k) \leq k^\alpha$ for all $k\geq 1$. 
	\end{enumeratei}
	In particular, for any $\alpha \in (0,1),$ $f(k) = k^{\alpha}, \, k \ge 1,$ is a sublinear attachment function.
	
\end{enumeratea}

The following terminology will be useful in the sequel. 

\begin{defn}
	\label{def:growing-trees}
	A collection of trees $\set{\cT(n):n\geq 1}$ is called a sequence of growing trees if the sequence starts with some finite tree $\cT(1)$ with a distinguished vertex called the root and then, for each subsequent $n$, $\cT(n+1)$ is obtained from $\cT(n)$ by adding a single new vertex to $\cT(n)$ using a single edge from this new vertex to a vertex $v\in \cT(n)$. 
\end{defn}
\subsection{Root finding algorithms and persistence}

 Let $\bT$ be the space of equivalence classes (upto isomorphisms) of finite unrooted and unlabelled trees.   For a labelled (rooted or unrooted) tree $\cT$, write $\cT^{\circ}$ for the isomorphism class of $\cT$ in $\bT$. For $\vt \in \bT$ and vertex $v\in \vt $ write $(\vt, v)$ for the tree rooted at $v$. For a rooted tree $(\vt, v)$ and a vertex $u\in \vt$ write $\desc{(\vt, v)}{u}$ to denote the subtree of $u$ and it's descendants (viewed as rooted tree with root $v$). Thus $\desc{(\vt, v)}{v} = (\vt,v)$.

 Now fix $K\geq 1$ and a mapping $H_K$ on $\bT$ that takes an input tree $\vt \in \bT$ and outputs a subset of $K$ vertices from $\vt$. 
 
 \begin{defn}[Root finding algorithm]
 	\label{def:root-find}
	Fix  $0<\eps < 1$ and $K\geq 1$. A mapping $H_K$ is called a budget $K$ root finding algorithm with error tolerance $\eps$ for the collection of random trees $\set{\cT_f(n):n\geq 2}$ if 
	\begin{equation}\label{rootinflim}
	\liminf_{n\to\infty} \pr(v_1 \in H_K(\cT_f(n)^\circ)) \geq 1-\eps.
	\end{equation}
 \end{defn}

One of the most natural mechanisms to generate mappings $H_K$ is via {\bf centrality} measures. Precisely given $\vt\in \bT$: 
 \begin{inparaenuma}
 	\item For each vertex $v\in \vt$ compute a specific measure of centrality $\Psi_{\vt}(v)$. Examples of such measures are given below.  This measure depends purely on the topology of the network as the standing assumption is that we do not have access to vertex labels or other information that could indicate age of vertices in the network. 
	\item For fixed $K$ let $H_{K,\Psi}(\vt)$ be the output of the ``top'' (could be largest or smallest depending on the centrality measure) $K$ vertices measured according to the above centrality measure.  Write $H_{K,\Psi}(\vt):= (v_{1,\Psi}(\vt), v_{2,\Psi}(\vt), \ldots, v_{K, \Psi}(\vt))$ (for the time being breaking ties arbitrarily). 
 \end{inparaenuma} 
 
 Two natural centrality measures are:
 
 \begin{enumeratei}
 	\item {\bf Degree centrality:} Here for $\vt$ and $v\in \vt $ let $\Psi_{\deg}(v)$ denote the degree of the vertex.  Thus in this case $H_{K,\Psi}(\vt)$ outputs the $K$ {\bf largest} degree vertices. Properties of such centrality measures for probabilistic models of dynamic networks were studied in for example \cite{DM,galashin2013existence,banerjee2020persistence}. 
	\item {\bf Jordan centrality \cite{jordan1869assemblages}:} This is the main centrality measure considered in this paper. For any finite set $A$ write $|A|$ for the cardinality of $A$. Now define 
	\begin{equation}
		\label{eqn:psi-def}
		\Psi_{\vt}(v):= \max_{u\in \vt\setminus \set{v}} |\desc{(\vt,v)}{u}|
	\end{equation}
	In words, viewing $v$ as the root of $\vt$,  $\Psi_{\vt}(v)$ is the largest size of the family trees of the children of $v$. We sometimes put $\vt$ in the subscript to emphasize the dependence on the underlying tree $\vt$.  Note that the above measure is {\bf maximized} at vertices of degree one. Thus in this case vertices more central in the tree would have smaller $\Psi$-values. A vertex that minimizes the above centrality measure is called a \emph{Jordan centroid}.  In this case  $H_{K,\Psi}(\vt)$ outputs the $K$ vertices with {\bf smallest} $\Psi$-values. 
 \end{enumeratei}
 
 A natural question then in the context of probabilistic models for evolving networks involves the notion of {\bf persistence}.  

 \begin{defn}[Persistence]\label{def:pers}
 	Fix $K\geq 1$ and a network centrality measure $\Psi$. For the random tree process $\set{\cT_f(n):n\geq 2}$ say that the sequence is {\bf $(\Psi,K)$ persistent} if $\exists$ $n^* < \infty$ a.s. such that for all $n\geq n^*$ the optimal $K$ vertices $(v_{1,\Psi}(\cT_f(n)^\circ), v_{2,\Psi}(\cT_f(n)^\circ), \ldots, v_{K, \Psi}(\cT_f(n)^\circ))$ remain the same and further the relative ordering amongst these $K$ optimal vertices remains the same. 
 \end{defn}
 Thus, for example, if $\Psi$ is the degree of a vertex (referred to as degree centrality) and $K=1$ then an evolving network sequence would be $(\Psi, 1)$ persistent if, almost surely, the identity of the maximal degree vertex fixates after finite time. Persistence of centrality measures both allows one to estimate the initial seed using these measures and validates robustness properties of these measures.
A weaker notion of persistence (specific to the Jordan centrality measure) was introduced in \cite{jog2016analysis} in their analysis of sub-linear preferential attachment schemes: 
 \begin{defn}[Terminal centroid \cite{jog2016analysis}]\label{def:term-cent}
 	For a sequence of growing trees $\set{\cT(n):n\geq 1}$, say that a vertex $v^*\in \cup_{n=1}^\infty \cT(n)$ is a terminal centroid if for every vertex $u\neq v^*\in \cup_{n=1}^\infty \cT(n)$ there exists a finite time $M:=M(u) <\infty$ such that for all $n\geq M(u)$, with $\Psi$ as in \eqref{eqn:psi-def},
	\[\Psi_{\cT(n)}(v^*)\leq \Psi_{\cT(n)}(u). \]
 \end{defn}
 Note that as described in \cite{jog2016analysis}, existence of a terminal centroid vertex $v^*$ does not imply persistence as in Definition \ref{def:pers} since $v^*$ might be a terminal centroid without being the centroid at any finite time $n$. Thus this notion of persistence is weaker than Definition \ref{def:pers}. 
 
\subsection{Informal description of our aims and results}

This paper has the following two major aims:

\begin{enumeratea}
	
\item {\bf Persistence of Jordan centrality and estimation of the root:} In the context of growing random trees, the Jordan centrality measure has been used in \cite{bubeck2017finding} to derive root finding algorithms when the underlying trees are grown using pure preferential attachment and uniform attachment. When the attachment function is sub-linear then existence of a \emph{terminal centroid} (see Definition \ref{def:term-cent}) was proven in \cite{jog2016analysis}.  The aim of this paper is to understand the performance of Jordan centrality for deriving root finding algorithms for random trees grown using general attachment functions $f$. We will establish  upper and lower bounds of the budget $K$ in terms of the error tolerance $\eps$. In this same setting, we establish persistence (as in Definition \ref{def:pers}) of the optimal vertices as measured via Jordan centrality thus ensuring robustness of these root finding procedures. 

\item {\bf Properties of branching process limits:} There is a deep connection between the random tree models in Section \ref{sec:model} and supercritical continuous time branching processes \cite{athreya1972,jagers-ctbp-book,jagers1984growth,nerman1981convergence}, also known as Crump-Mode-Jagers (CMJ) branching processes, see Lemma \ref{lem:ctb-embedding-no-cp}. A fundamental result in the theory of branching processes is the almost sure convergence of the appropriately normalized population size to a limiting random variable $W_{\infty}$ (see Theorem \ref{prop:convg-limit}). For specific models such as the uniform attachment scheme or the pure preferential attachment scheme, one can compute the limit distributions exactly and these play a major role in understanding the performance of root-finding algorithms, see e.g. \cite{bubeck2017finding}. Much less is known about these limit random variables for general $f$ beyond continuity properties. The following two technical contributions in this paper lie at the heart of root detection algorithms for general $f$, and are of independent interest.

Firstly, under some assumptions, we show that $W_{\infty}$ has exponential tails. This result plays a major role in the analysis of root-finding algorithms. In particular, it is used to obtain precise quantitative bounds on the probability that the number of descendants of an individual $v$ in the branching process exceeds the total number of descendants of a collection $\mathcal{S}$ of individuals (not containing $v$) for any large time in terms of the size of $\mathcal{S}$. This naturally translates to analogous estimates for growing random trees by their connection to continuous time branching processes.

In addition, we also provide convergence rates for the normalized population size to $W_{\infty}$ in an appropriate sup-norm metric (Theorem \ref{thm:expon-convergence} and Corollary \ref{cor:sup-bound}). These results are used to furnish quantitative estimates for the probability that, for some $K \in \mathbb{N}$, the Jordan centrality score of $v_n$ takes one of the $K$ smallest values in $\cT(j)$ for any $j \ge n$, for sufficiently large $n$. An application of the Borel Cantelli Lemma then establishes persistence of the $K$ most central vertices (as measured by Jordan centrality) in the growing network for any fixed $K \in \mathbb{N}$ (Theorem \ref{thm:persistence}). 


\end{enumeratea}

\subsection{Organization of the paper}
We start by defining fundamental objects required to state our main results in Section \ref{sec:not}. Our main results are in Section \ref{sec:main-res}. In Section \ref{sec:disc} we discuss the relevance of this work as well as related literature. The remaining sections are devoted to proofs of the main results.  
\section{Preliminaries}
\label{sec:not}

\subsection{Mathematical notation}

We use $\stod$ for stochastic domination between two real valued probability measures. For $J\geq 1$ let $[J]:= \set{1,2,\ldots, J}$. If $Y$ has an exponential distribution with rate $\lambda$, write this as $Y\sim \exp(\lambda)$. Write $\bZ$ for the set of integers, $\bR$ for the real line, $\bN$ for the natural numbers and let $\bR_+:=(0,\infty)$. Write $\convas,\convp,\convd$ for convergence almost everywhere, in probability and in distribution respectively. For a non-negative function $n\mapsto g(n)$,
we write $f(n)=O(g(n))$ when $|f(n)|/g(n)$ is uniformly bounded, and
$f(n)=o(g(n))$ when $\lim_{n\rightarrow \infty} f(n)/g(n)=0$.
Furthermore, write $f(n)=\Theta(g(n))$ if $f(n)=O(g(n))$ and $g(n)=O(f(n))$.
Finally, we write that a sequence of events $(A_n)_{n\geq 1}$
occurs \emph{with high probability} (whp) when $\pr(A_n)\rightarrow 1$ as $n \rightarrow \infty$. To ease notation, for the rest of this article, we will write $\cT^\circ(n)$ for $(\cT_f(n))^\circ$, $n \ge 2$.

\subsection{Assumptions on attachment functions}
Here we setup constructions needed to state the main results.  We will need the following assumption on the attachment functions of interest in this paper. We mainly follow \cite{jagers-ctbp-book,jagers1984growth,nerman1981convergence,rudas2007random}. 

\begin{ass}
	\label{ass:attach-func} 
	\begin{enumeratei}
		\item Every attachment function $f$ is assumed to satisfy $f_* := \inf_{i\geq 1} f(i) > 0$. 
		\item Every attachment function $f$ can grow at most linearly i.e. $\exists C< \infty$ such that $\limsup_{k\to\infty} f(k)/k \leq C$. This is equivalent to there existing a constant $C$ such that $f(k)\leq Ck$ for all $k\geq 1$. 
		\item Consider the following function $\hat{\rho}:(0,\infty)\to (0,\infty]$ defined via,
\begin{equation}
\label{eqn:rho-hat-def}
	\hat{\rho}(\lambda):= \sum_{k=1}^\infty \prod_{i=1}^{k} \frac{f(i)}{\lambda + f(i)}. 
\end{equation}
 Define 
$\underline{\lambda}:= \inf\set{\lambda > 0: \hat{\rho}(\lambda) < \infty}$. 
We assume,
\begin{equation}
\label{eqn:prop-under-lamb}
\underline{\lambda} < \infty \ \text{ and } \ \lim_{\lambda\downarrow\underline{\lambda}} \hat{\rho}(\lambda) > 1. 
\end{equation}
\end{enumeratei}
\end{ass}
Using (iii) of the above Assumption and the monotonicity of $\hat{\rho}(\cdot)$, there exists a unique $\lambda^*:=\lambda^*(f) \in (0, \infty)$ such that 
\begin{equation}
\label{eqn:malthus-def}
	\hat{\rho}(\lambda^*) = 1. 
\end{equation}
This object is often referred to as the Malthusian rate of growth parameter. The above assumptions are relatively standard and are required for well-definedness of an associated branching process in the next section. The next assumption is less standard. 

\begin{ass}
	\label{ass:lim-sup}
	With $\lambda^*$ denoting the Malthusian rate of growth parameter,  $\limsup_{i\to\infty} f(i)/i < \lambda^*$. In other words, $\hat{\rho}\left(\limsup_{i\to\infty} f(i)/i \right) > 1$.
\end{ass}

Assumption \ref{ass:lim-sup} is algebraic in nature and can be verified for a given attachment function $f$ via the explicit representation of $\hat{\rho}(\cdot)$ given in \eqref{eqn:rho-hat-def}. The following lemma gives some general settings where this assumption is satisfied.

\begin{lemma}\label{lem:limfi-i-suff}
	Assume $f$ satisfies Assumption \ref{ass:attach-func}. Assumption \ref{ass:lim-sup} is satisfied if any of the following conditions hold:
	\begin{itemize}
	\item[(i)] There exists $C^*\geq 0$ such that $\lim_{i\to\infty} f(i)/i = C^*$;
	\item[(ii)] $\limsup_{i\to\infty} f(i)/i < f_*$.
	\end{itemize}  
	In particular uniform attachment, (affine) preferential attachment and sublinear preferential attachment all satisfy Assumption \ref{ass:lim-sup}. 
\end{lemma}

\begin{proof}
If (i) holds, then Assumption \ref{ass:lim-sup} follows from \cite[Lemma 8.1]{banerjee2018fluctuation}. For the second assertion, observe that
$$
1=\hat{\rho}(\lambda^*) = \sum_{k=1}^\infty \prod_{i=1}^{k} \frac{f(i)}{\lambda^* + f(i)} \ge \sum_{k=1}^\infty\left(\frac{f_*}{\lambda^* + f_*}\right)^k = \frac{f_*}{\lambda^*}.
$$
Therefore, if (ii) holds, then $\lambda^* \ge f_* > \limsup_{i\to\infty} f(i)/i$.
\end{proof}

\begin{rem}
If $\sum_{i=1}^{\infty}{1}/{f(i)^2} < \infty$, Assumption \ref{ass:lim-sup} implies \eqref{eqn:prop-under-lamb} \cite[Lemma 7.8]{banerjee2020persistence}. 
\end{rem}

\subsection{Branching processes}
\label{sec:br-def}
Fix an attachment function $f$ as above. We can construct a point process $\xi_f$ on $\bR_+$ as follows: Let $\set{E_i:i\geq 1} $ be a sequence of independent exponential random variables with $E_i \sim \exp(f(i))$. Now define 
$\sigma_i:= \sum_{j=1}^{i} E_i$ for $ i\geq 1 $ with $\sigma_0 =0$.
The point process $\xi_f$ is defined via, 
\begin{equation}
\label{eqn:xi-f-def}
	\xi_f:=(\sigma_1, \sigma_2, \ldots).
\end{equation}
Abusing notation, we write for $t\geq 0$,
\begin{equation}
\label{eqn:xi-f-t}
	\xi_f[0,t]:= \#\set{i: \sigma_i \leq t}, \qquad \mu_f[0,t]:= \E(\xi_f[0,t]). 
\end{equation}
$\xi_f, \mu_f$ can be naturally extended to measures on $(\bR_+, \cB(\bR_+))$. We will need a variant of the above objects: for fixed $k\geq 1$, let $\xi_f^{\sss(k)}$ denote the `shifted' point process where the first inter-arrival time is $E_k$. Namely, define the sequence,
\[\sigma_{{\sss(k)},i} = E_k + E_{k+1}+\cdots E_{k+i-1}, \qquad i\geq 1.\]
As before let $\sigma_{{\sss(k)},0} = 0$.  Then define,
\begin{equation}
\label{eqn:xi-k-f-t-def}
	\xi_f^{\sss(k)}:=(\sigma_{\sss(k), 1}, \sigma_{\sss(k), 2}, \ldots),  \qquad \mu_f^{\sss(k)}[0,t]:= \E(\xi_f^{\sss(k)}[0,t]). 
\end{equation}
As above, $\xi_f^{\sss(k)}[0,t]:= \#\set{i: \sigma_{\sss(k), i} \leq t}$. We abbreviate $\xi_f[0,t]$ as $\xi_f(t)$ and similarly $\mu_f(t), \xi_f^{\sss(k)}(t),\mu_f^{\sss(k)}(t)$. Note that $\xi_f^{\sss(1)}(\cdot) = \xi_f(\cdot)$. 


\begin{defn}[Continuous time branching process (CTBP)]
	\label{def:ctbp}
	Fix attachment function $f$ satisfying Assumption \ref{ass:attach-func}(ii). A continuous time branching process driven by $f$, written as $\{\BP_f(t) : t \ge 0\}$, is defined to be a branching process started with one individual at time $t=0$ and such that this individual, as well as every individual born into the system, has an offspring distribution that is an independent copy of the point process $\xi_f$ defined in \eqref{eqn:xi-f-def}.  
\end{defn}
We refer the interested reader to \cite{jagers-ctbp-book,athreya1972,nerman1981convergence} for general theory regarding continuous time branching processes. The branching processes under consideration also go by the name of Crump-Mode-Jagers (CMJ) branching processes. We will also use $\BP_f(t)$ to denote the collection of all individuals at time $t\geq 0$.  We will use $Z_f(t)$ or $|\BP_f(t)|$ to denote the number of individuals in the population at time $t$. Note in our construction, by our assumption on the attachment function, individuals continue to reproduce forever.  Write $m_f(\cdot)$ for the corresponding expectation i.e.,  
\begin{equation}
\label{eqn:mft-def}
	m_f(t):= \E(Z_f(t)), \qquad t\geq 0,
\end{equation} 
Under Assumption \ref{ass:attach-func}(ii), it can be shown \cite[Chapter 3]{jagers-ctbp-book} that for all $t>0$, $m_f(t) <\infty$ and $m_f(\cdot)$ is strictly increasing with $m_f(t)\uparrow\infty$ as $t\uparrow\infty$.

The connection between CTBP and the discrete random tree models in the previous section is given by the following result which is easy to check using properties of exponential distribution (and is the starting point of the Athreya-Karlin embedding \cite{athreya1968}). Start with two vertices $v_1$ and $v_2$ connected by a single edge. Now let $\BP^{\sss(1 \downarrow)}_f(\cdot), \BP^{\sss(2 \downarrow)}_f(\cdot)$ be two independent branching processes driven by  $\xi_f$ as in Definition \ref{def:ctbp} started from $v_1$ and $v_2$ respectively. Define $\widetilde{\BP}_f(t) := \BP_f^{\sss(1 \downarrow)}(t) \cup \BP_f^{\sss(2 \downarrow)}(t)$. {For $m\geq 3$ label the vertices in $\widetilde{\BP}_f(\cdot)$ in order of their arrival into the system. Thus, for any $t\ge 0$, $\widetilde{\BP}_f(t)$ can be viewed as a labelled tree rooted at $v_1$, with edges between parents and their offspring. Write the natural filtration generated by both branching processes as $\set{\cF(t):t\geq 0}$, where $\cF(t) := \sigma\set{(\BP_f^{\sss(1 \downarrow)}(s), \BP^{\sss(2 \downarrow)}_f(s)): s \le t}, \ t \ge 0$.} Define the sequence of stopping times:
\begin{equation}
\label{eqn:tm-stop-def}
	\eta_m^{\sss(1,2)}:=\inf\set{t\geq 0: |\widetilde{\BP}_f(t)| = m}, \qquad m\geq 2, 
\end{equation} 
{where $|\widetilde{\BP}_f(t)|$ denotes the number of individuals in $\widetilde{\BP}_f(t)$ at time $t$.}
In particular, by the labelling convention described above, the vertex arriving at time $\eta_m^{\sss(1,2)}$ is labelled as $v_m$. Here we have added both $1,2$ in the superscript to make it clear that this stopping time depends on both branching processes. 
\begin{lemma}\label{lem:ctb-embedding-no-cp}
	Fix attachment function $f$ consider the sequence of random trees $\set{\cT_f(m): m\geq 2 }$ constructed using attachment function $f$ rooted at $v_1$. Consider the continuous time construction $\{\widetilde{\BP}_f(t):t\geq 0\}$ as above. Then viewed as a sequence of growing random labelled rooted trees we have, $\{\widetilde{\BP}_f(\eta_m^{\sss(1,2)}): m\geq 2\} \stackrel{d}{=} \{\cT_f(m):m\geq 2\}.$
\end{lemma}

\begin{rem}
Note that multiplying the function $f$ by a scalar speeds up (or slows down) the continuous time branching process, but the law of the embedding $\{\widetilde{\BP}_f(\eta_m^{\sss(1,2)}): m\geq 2\}$ remains unaltered (as the associated stopping times also change). This is clear in the dynamics of the discrete time tree process as the attachment probabilities involve ratios of quantities involving $f$, which do not change on multiplying $f$ by a scalar.
\end{rem}

\section{Main Results}
\label{sec:main-res}

\subsection{Bounds for root finding algorithms and persistence}
\label{sec:res-bds-root-pers}

For the rest of this Section $\Psi$ will denote the Jordan centrality measure as in \eqref{eqn:psi-def} and for a fixed budget $K$, $H_{K,\Psi}(\cdot)$ will denote the corresponding Jordan centrality measure based root finding algorithm. For a given error tolerance $\eps$, let $K_{\Psi}(\eps)$ denote the smallest budget $K$ such that $H_{K,\Psi}(\cdot)$ is a root finding algorithm with error tolerance $\eps$ for the collection of random trees $\set{\cT_f(n): n\geq 2}$ (in the sense of Definition \ref{def:root-find}). If there is no finite $K$ satisfying \eqref{rootinflim}, we set $K_{\Psi}(\eps)= \infty$. The first theorem gives upper bounds on this budget $K_{\Psi}(\eps)$ to ensure recoverability of the root.   

\begin{theorem}[Budget sufficiency bounds] 
	\label{thm:upp-root-find}
	Suppose the attachment function $f$ satisfies Assumptions \ref{ass:attach-func} and \ref{ass:lim-sup}. 
	\begin{enumeratea}
		\item Suppose for some $\overline{C}_f>0$, $\beta \ge 0$, $f$ satisfies $f_{*} \leq f(i) \leq \overline{C}_f \cdot i + \beta$ for all $i\geq 1$. Then $\exists$ positive constants $C_1, C_2$ such that for any error tolerance $0< \eps <1$, the budget requirement satisfies, 
		\[K_{\Psi}(\eps)\leq \frac{C_1}{\eps^{(2\overline{C}_f+\beta)/f_{*}}}\exp(\sqrt{C_2\log{1/\eps}}).\]
		\item If further the attachment function $f$ is in fact bounded with $f(i)\leq f^{*}$ for all $i \ge 1$ then one has for any error tolerance $0< \eps <1$,
		\[K_{\Psi}(\eps)\leq \frac{C_1}{\eps^{f^*/f_{*}}}\exp(\sqrt{C_2\log{1/\eps}}).\] 
	\end{enumeratea}
\end{theorem}

A natural question then is if there are qualitatively similar lower bounds for the budget if one does use Jordan centrality. This is the focus of the next result. 

\begin{theorem}[Budget necessary bounds] 
	\label{thm:low-root-find}
	Suppose the attachment function $f$ satisfies Assumptions \ref{ass:attach-func} and \ref{ass:lim-sup}. 
	\begin{enumeratea}
		\item If $\exists$ $\underline{C}_f> 0$ and $\beta\geq 0$ such that $f(i)\geq \underline{C}_f \cdot i + \beta$ for all $i\geq 1$ then $\exists$ a positive constant $C_1^\prime$ such that for any error tolerance $0< \eps <1$, 
		\[K_{\Psi}(\eps) \geq \frac{C_1^\prime}{\eps^{(2\underline{C}_f+\beta)/f(1)}}.\]
		\item For general $f$ one has for any error tolerance $0< \eps <1$,
	\[K_{\Psi}(\eps) \geq \frac{C_1^{\prime}}{\eps^{f_*/f(1)}}.\]
	\end{enumeratea}
	   \end{theorem}
Applying the above results for the special cases described in Section \ref{sec:model} result in the following bounds. 
\begin{corollary}[Special cases]
	For specific attachment functions the budget requirements for root finding algorithms using Jordan centrality satisfy:
	\label{cor:root-find-special-cases}
	\begin{enumeratea}
		\item {\bf Uniform attachment:} $\frac{C_1^\prime}{\eps}\leq K_{\Psi}(\eps) \leq \frac{C_1}{\eps} \exp(\sqrt{C_2\log{\frac{1}{\eps}}})$. 
		\item {\bf Pure Preferential attachment:} 
		\[\frac{C_1^\prime}{\eps^{2}} \leq K_{\Psi}(\eps) \leq \frac{C_1}{\eps^2} \exp(\sqrt{C_2\log{\frac{1}{\eps}}}) . \]
		\item {\bf Affine preferential attachment:} 
		\[\frac{C_1^\prime}{\eps^{\frac{2+\beta}{1+\beta}}} \leq K_{\Psi}(\eps) \leq \frac{C_1}{\eps^{\frac{2+\beta}{1+\beta}}} \exp(\sqrt{C_2\log{\frac{1}{\eps}}}) . \]
		\item {\bf Sublinear preferential attachment:}
		\[\frac{C_1^\prime}{\eps} \leq K_{\Psi}(\eps) \leq \frac{C_1}{\eps^2} \exp(\sqrt{C_2\log{\frac{1}{\eps}}}) . \]
	\end{enumeratea}
\end{corollary}

The above Corollary follows by directly applying Theorems \ref{thm:upp-root-find} and \ref{thm:low-root-find}.  

\begin{rem}
	While (b) in the above Corollary is a special case of (c), we included the additional case to highlight the sensitive dependence on $\beta$ of the above bounds. In particular as $\beta \to \infty$ upto first order $K_{\Psi}(\eps) \asymp \eps^{-1}$ which is the same as one has in the uniform attachment case.  
\end{rem}

The next result establishes robustness of the estimates using Jordan centrality. For centrality measure $\Psi$ and $K\geq 1$, recall the notion of $(\Psi, K)$ persistence in Definition \ref{def:pers}.  

\begin{theorem}[Persistence]
	\label{thm:persistence}
	Let $\Psi$ denote the Jordan centrality as in \eqref{eqn:psi-def}.  Under Assumptions \ref{ass:attach-func} and \ref{ass:lim-sup}, for any $K\geq 1$, the sequence of growing random trees $\set{\cT_f(n):n\geq 1}$ is $(\Psi, K)$ persistent. 
\end{theorem}  

\subsection{Branching process limits}
\label{sec:res-bp-limits}

Implicitly or explicitly, the main tool in the analysis of the Jordan centrality for growing random tree models are limits of appropriately normalized sizes of the associated branching processes. More precisely standard results about such branching processes \cite{jagers1984growth} imply the following result. 

\begin{theorem}\label{prop:convg-limit}
	Under Assumptions \ref{ass:attach-func} and \ref{ass:lim-sup} there exists a strictly positive random variable $W_\infty$ such that with $\lambda^*$ denoting the Malthusian rate of growth parameter \eqref{eqn:malthus-def},
	\begin{equation}
	\label{eqn:bp-convg-winf}
		 e^{-\lambda^* t} Z_f(t) \stackrel{a.s., \bL^2}{\longrightarrow} W_\infty, \qquad \mbox{ as } t\to\infty.
	\end{equation} 
	Further $W_\infty$ is an absolutely or singular continuous random variable supported on all of $\bR_+$. 
\end{theorem} 
The following technical results about the branching process form the core of the proofs of the previous two sections and are of independent interest. 

\begin{theorem}[Exponential tails]\label{thm:exponential-mom}
	Under Assumptions \ref{ass:attach-func} and \ref{ass:lim-sup} the branching process limit $W_\infty$ in Theorem \ref{prop:convg-limit} has a finite moment generating function in an open interval about the origin. 
\end{theorem}

While Theorem \ref{prop:convg-limit} provides asymptotic information, the following two results quantify rates of convergence in this result and form the technical workhorse for the proof of Theorem \ref{thm:persistence}. 

\begin{theorem}\label{thm:expon-convergence}
	Under Assumptions \ref{ass:attach-func} and \ref{ass:lim-sup}, $\exists$ finite constants $C, C^\prime, \delta$ (depending on $f$) such that for any $t\geq 0$,
	\[\pr\left[\sup_{s\in [t,\infty)}\left|e^{-\lambda^*s}Z_f(s) - W_\infty\right| > e^{-\delta t}\right] \leq Ce^{-C^\prime t}. \]
\end{theorem}

\begin{corollary}\label{cor:sup-bound}
	Under the hypothesis of Theorem \ref{thm:expon-convergence} there exists $C <\infty$ such that for any $A>0$,
	\[\pr\left(\sup_{s\geq 0} e^{-\lambda^* s} Z_f(s) \geq A\right) \leq C\log{A}/A^2. \]
\end{corollary}

\begin{rem}
A natural object of interest that arises out of Theorem \ref{thm:persistence} is the random time $\Gamma := \inf\{n \in \mathbb{N} : v_{1,\Psi}(\cT_f(\ell)^\circ) = v_{1,\Psi}(\cT_f(n)^\circ) \ \forall \ \ell \ge n\},$ namely, the first time after which the identity of the centroid remains the same as the tree grows. We believe that the rate of convergence result Theorem \ref{thm:expon-convergence} will play a key role in analyzing this time as it gives a way to compare the appropriately scaled tree sizes to their (random) limits. However, a thorough analysis of $\Gamma$, including moment bounds, tail asymptotics, etc., will require a more explicit quantification of the constants $C, C'$ appearing in Theorem \ref{thm:expon-convergence}. Moreover, one will need more detailed information on the distributional properties of $W_{\infty}$ beyond the ones used in this article (finiteness of exponential moments and continuity properties). We will explore this in a subsequent article.
\end{rem}

\section{Discussion}
\label{sec:disc}

\subsection{Related work}
The general area of ``network archaeology'' \cite{navlakha2011network} has witnessed enormous growth over the past decade. In the context of the probability community, the papers closest to this paper are \cite{shah2011rumors,shah2012rumor,bubeck2017finding,jog2016analysis,jog2018persistence}. Identifying the spread of ``rumors'' in networks was considered in \cite{shah2011rumors,shah2012rumor} where a ``rumor centrality'' measure was derived which later turned out \cite{bubeck2017finding} to have close connections to Jordan centrality as well as maximum-likelihood based schemes for estimating the root. Results similar to Theorems \ref{thm:upp-root-find} and \ref{thm:low-root-find} were first obtained in \cite{bubeck2017finding} for the uniform attachment and pure preferential attachment schemes. In these special cases, the limit random variables for subtree sizes (or analogously the continuous time branching process limits) can be explicitly derived in terms of Gamma and Dirichlet distributions which allows a direct analysis of root finding algorithms based on Jordan centrality. Using a careful combinatorial analysis of the likelihood function, \cite{bubeck2017finding} were even able to derive lower bounds required for {\bf any} root finding algorithm in the context of the uniform and pure preferential attachment models (in particular showing that for the uniform model $K(\eps)$ for any root finding algorithm with error tolerance $\eps$ has to satisfy $K(\eps) \geq \exp(\sqrt{\frac{1}{30}\log \frac{1}{2\eps}})$). Further advances related to the uniform attachment scheme can be found in \cite{lugosi2019finding,devroye2018discovery}. For related questions on the influence of the initial seed on the distribution of trees grown using uniform or pure preferential attachment see \cite{bubeck-mossel,curien2014scaling,racz2020correlated}. 

The two papers most influential for this paper are \cite{jog2016analysis,jog2018persistence} which amongst various advances derived fundamental deterministic properties of the Jordan centrality measure for growing trees (see Section \ref{sec:proofs-det-jordan} for some examples). Further \cite{jog2016analysis} was the first paper (that we are aware of) that used embeddings into continuous time branching processes to analyze Jordan centrality for sublinear preferential attachment models (see Section \ref{sec:model}). For this class of models they were able to show existence of a terminal centroid (see Definition \ref{def:term-cent}) but were unable to show the existence of persistence of Jordan centrality. Part of the reason was the lack of rates of convergence results for continuous time branching processes namely Theorems \ref{thm:expon-convergence} and Corollary \ref{cor:sup-bound} in this paper. These results had to wait for the technical foundations laid in \cite{banerjee2018fluctuation} where rates of convergence results were needed in the context of change point detection problems for evolving networks. Moreover, although the existence of a $K(\eps)$ for any $\eps \in (0,1)$ was shown in \cite{jog2016analysis} in the context of sublinear preferential attachment models, no explicit algebraic form was obtained. Here we connect $K(\eps)$ to the tail behavior of the random variable $W_{\infty}$ and use its exponential tails (established in Theorem \ref{thm:exponential-mom}) to obtain explicit upper and lower bounds on $K(\eps)$ for general attachment functions $f$. These bounds, in particular, coincide for uniform and affine linear preferential attachment models, strengthening the bounds obtained in \cite{lugosi2019finding}.

\subsection{Persistence of other centrality measures}
The other major centrality measure for which questions such as persistence have been explored is degree centrality. As described in Definition \ref{def:root-find}, in this case for any budget $K$, one would output the $K$ vertices with maximal degree (the so called ``hubs'' of the network). Sufficient conditions for persistence of maximal degree vertices under convexity assumptions of the attachment function were obtained in \cite{galashin2013existence} whilst more general sufficient conditions for a related evolving network model were derived in \cite{DM}. Most closely related to this paper is \cite{banerjee2020persistence} where it was shown using martingale concentration techniques and continuous time branching processeses that for a general attachment function $f$, the condition $\sum_{i=1}^{\infty} 1/(f(i))^2 < \infty$ (coupled with additional technical conditions) was necessary and sufficient for the emergence of persistent hubs. In the regimes of existence of persistent hubs, degree centrality (which is computationally much more efficient and extends to more general networks than trees) gives an alternative estimation procedure for the root. We plan to report on this in future research.

\subsection{Related questions in probabilistic combinatorics}
The other major area where notions similar to persistence have been explored has been in the context of the evolution of the connectivity structure in random graphs as first formulated by Erd\H{o}s.  Consider the \erdos random graph process described as follows:  start with $n$ vertices with no edges. At each stage an edge is selected uniformly at random amongst all possible ${n \choose 2}$ edges and added to the system. The objects of interest now are the {\bf sizes} (i.e. the number of vertices) of connected components. Erd\H{o}s suggested viewing this process as a ``race'' between components in growing their respective sizes. Call the largest connected component at any given time as the {\bf leader}.  One of the questions raised by Erd\H{o}s was to understand asymptotics for the leader time: the time beyond which the identity of the leader does not change, thus {\bf the leader persists} for all periods of connectivity after this time. Rigorous results for asymptotics for this so called ``leader problem'' were derived in \cite{luczak1990component,addario2017probabilistic}.

\subsection{Proof outline}
The recurring theme in our proofs is an interplay between the discrete time attachment tree process $\{\cT_f(n) : n \ge 2\}$ and its continuous time embedding described in Section \ref{sec:br-def}. This connection has been exploited before in the context of asymptotics related to \emph{local weak convergence} \cite{rudas2007random}, where one looks at asymptotic properties in a local neighborhood of a uniformly chosen vertex in $\cT_f(n)$ as $n \rightarrow \infty$. There one can rely on `softer' properties like almost sure convergence of associated statistics in the CTBP to derive results for the attachment tree process. However, for analyzing centrality measures, one needs quantitative information on the \emph{rate of convergence} of these statistics to their (random) limits and \emph{tail behavior} of the limiting quantities. The key technical contributions of this article comprise obtaining rates of convergence in sup-norm for the normalized population size of the CTBP to its random limit $W_{\infty}$ (Theorem \ref{thm:expon-convergence}) and showing existence of exponential moments of the limiting random variable $W_{\infty}$ (Theorem \ref{thm:exponential-mom}). Theorem \ref{thm:expon-convergence} is shown by extending techniques of \cite{banerjee2018fluctuation}. Theorem \ref{thm:exponential-mom} is shown by devising careful inductive arguments to estimate moments of an associated random variable $Y$ defined in \eqref{eqn:y-def} and then using results on solutions of recursive distributional equations. 

For showing persistence (Theorem \ref{thm:persistence}), the work of \cite{jog2016analysis,jog2018persistence} illustrates that comparing $\Psi$ values of two vertices $u$ and $v$ reduces to \emph{comparing sizes of associated subtrees rooted at $u$ and $v$} (see Lemma \ref{lem:jord-cent-det}). Theorem \ref{thm:expon-convergence} is then used to quantitatively estimate appropriately normalized subtree sizes by their random limits. This and an application of the Borel Cantelli Lemma yield Theorem \ref{thm:persistence}. For deriving the budget sufficiency bounds, concentration arguments are used to show that the sizes of the two evolving subtrees obtained by deleting the edge between $v_1$ and $v_2$ do not become `very disproportionate' in the sense that, with high probability, one is at least a suitably large multiple $M$ of the logarithm of the other \emph{uniformly over all large times} (Lemma \ref{lem:no-too-small}). This is then used to show that if the $i$-th added vertex has to ever attain a smaller $\Psi$ value than the root, the size of a subtree attached to it has to beat the collective size of approximately $M\log i$ independent subtrees attached to other vertices at some future time. Quantifying the probability of this happening is key to obtaining the explicit estimates in Theorem \ref{thm:upp-root-find} and uses the exponential moment bound in Theorem \ref{thm:exponential-mom} in a crucial way. For the budget necessary bounds, it is shown that, for any integer $K \ge 2$, the probability of the root not having one of the $\lfloor K/4\rfloor$ smallest $\Psi$ values can be bounded below (up to a multiplicative constant) by the probability that the \emph{root is a leaf in $\cT_f(K)$} (Proposition \ref{prop:leaf-lower}). Lower bounding the latter probability leads to the bounds in Theorem \ref{thm:low-root-find}.

\section{Proofs: Branching process asymptotics}
\label{sec:proofs-bp}

For the rest of the proof, to ease notation we write $\lambda$ instead of $\lambda^*$ for the Malthusian rate of growth as in \eqref{eqn:malthus-def}. 

\subsection{Branching process preliminaries}
\label{sec:bp-prelim}
We start by collecting some simple properties of branching processes that will be useful in the sequel. We start by describing one of the most famous examples of branching processes in continuous time. 

\begin{defn}[Rate $\nu$ Yule process]
\label{def:yule-process}
    Fix $\nu > 0$. A rate $\nu$ Yule process is a pure birth process $\set{Y_\nu(t):t\geq 0}$ with $Y_\nu(0)=1$ and where the rate of birth of new individuals is proportional to size of the current population. More precisely, $\pr(Y_\nu(t+dt) - Y_\nu(t) = 1|\cF(t)):= \nu Y_\nu(t) dt + o(dt)$ and $\pr(Y_\nu(t+dt) - Y_\nu(t) \ge 2|\cF(t)):= o(dt),$ where $\set{\cF(t):t\geq 0}$ is the natural filtration of the process. 
\end{defn}
There is a close connection between Yule processes and continuous time branching process. Precisely, consider a continuous time branching process as in Definition \ref{def:ctbp} with attachment function $f(i)\equiv \nu$ for all $i$. Write $\yu_\nu(\cdot)$ for this process.  Then it is easy to check that the process describing the size of the population $\set{|\yu_\nu(t)|:t\geq 0}  \stackrel{d}{=} \set{Y_\nu(t):t\geq 0}$. For the rest of this paper, when we use the term Yule process, we will refer to this associated branching process. 

The following is a standard property of the Yule process. 

\begin{lem}[{\cite[Section 2.5]{norris-mc-book}}]
\label{lem:yule-prop}
Fix $t >0$ and rate $\nu > 0$. Then $Y_\nu(t)$ has a Geometric distribution with parameter $p=e^{-\nu t}$. Precisely, $\text{ }\pr(Y_\nu(t) = k) = e^{-\nu t}(1-e^{-\nu t})^{k-1}, k\geq 1.$ The process $\set{Y_\nu(t)\exp(-\nu t):t\geq 0}$ is an $\bL^2$ bounded martingale and thus $\exists~ W>0 $ such that  $Y_\nu(t)\exp(-\nu t)\convas W$. Further $W = \exp(1)$. 
\end{lem} 

Recall that we assume that our attachment functions $f$ grow at most linearly (i.e. $f(i)\leq \overline{C}_f i$) for some constant $\overline{C}_f$. This motivates the next special case. 

\begin{defn}[$\LPA_{C,\beta}$ process]
	Fix $C> 0$ and $\beta \geq 0$. The continuous time branching processes as in Definition \ref{def:ctbp} with $f(i) = C\cdot i +\beta$ will be referred to as the Linear preferential attachment process and will be denoted by $\LPA_{C,\beta}(\cdot)$. 
\end{defn}

The following result easily follows from properties of the Exponential distribution. 

\begin{prop}
	\label{prop:bpf1-2pdom}
	\begin{enumeratea}
		\item Suppose $f_1, f_2$ are two attachment functions with $f_1(i)\leq f_2(i)$ for all $i\geq 1$. Then $\exists$ a coupling of the branching processes $\set{\BP_{f_1}(t):t\geq 0}$ and $\set{\BP_{f_2}(t):t\geq 0}$ on the same space so that $|\BP_{f_1}(t)| \leq |\BP_{f_2}(t)|$ for all $t\geq 0$. 
		\item Let the attachment function $f_*\leq f(i)\leq \overline{C}_f\cdot i+\beta$ $\forall$ $i\geq 1$.  Recall the continuous time construction in Definition \ref{def:ctbp} with two continuous time branching processes $\BP_f^{\sss(i \downarrow)}$ started from $v_1$ and $v_2$ with an edge between $v_1$ and $v_2$. We can couple $\BP^{\sss(1 \downarrow)}_f$ with a linear preferential attachment process $\LPA_{\overline{C}_f, \beta}$ and independently couple $\BP^{\sss(2 \downarrow)}_f$ with a Yule process $\yu_{f_*}$ such that 
		\[|\BP^{\sss(1 \downarrow)}_f(t)| \leq |\LPA_{\overline{C}_f, \beta}(t)|, \qquad |\BP^{\sss(2 \downarrow)}_f(t)| \geq |\yu_{f_*}(t)|, \qquad t\geq 0.  \]
	\end{enumeratea}
	
\end{prop}

The following is a simple example of the use of the above result and lower bounding Yule processes. 

\begin{lemma}\label{lem:zft-lbd-1}
	Assume $f$ satisfies $f_* := \inf_{i\geq 0} f(i)>0$. Let $Z_f(t) = |\BP_f(t)|$. Then for any $t\geq 0$ and $\gamma >0$, 
	\[\pr(Z_f(t) \leq e^{tf_*/2}) \leq e^{-tf_*/2}, \qquad \E([Z_f(t)]^{-\gamma})\leq  e^{-tf_*/2} + e^{-\gamma f_*t/2}. \]
\end{lemma}
\begin{proof}
	Consider the setting of Proposition \ref{prop:bpf1-2pdom} (a) with $f_1(i) \equiv f_*$ for all $i$ and $f_2 = f$. In this case $Z_{f_1}(\cdot)$ is a rate $f_*$ Yule process. Thus Proposition \ref{prop:bpf1-2pdom} (a) coupled with Lemma \ref{lem:yule-prop} implies by the union bound, 
	\[\pr(Z_f(t) \leq e^{tf_*/2}) \leq \pr(Z_{f_1}(t) \leq e^{tf_*/2}) \leq e^{-tf_*/2}.  \]
	This proves the first assertion of the lemma. The second assertion follows from the first assertion and the bound,  $\E([Z_f(t)]^{-\gamma}) \leq \pr(Z_f(t) \leq e^{tf_*/2}) + e^{-\gamma f_*t/2}$. 
\end{proof}

Now define the probability mass function, 
\begin{equation}
\label{eqn:pk-def}
    p_k=p_k(f) := \frac{\lambda}{\lambda + f(k)} \prod_{j=1}^{k-1}\frac{f(i)}{\lambda + f(i)}, \qquad k\geq 1. 
\end{equation}
Here we use the convention $\prod_{j=1}^0 = 1$. The fact that $\{p_k : k \ge 1\}$ is indeed a pmf follows from the definition of the Malthusian rate \eqref{eqn:malthus-def}.
For any $k\geq 0$ and $t\geq 0$, let $D(k,t)$ denote the number of vertices with $k$ children (alternatively out-degree $k$) in $\BP_f(t)$ and let $D_n(k)$ denote the number of vertices with out-degree $k$ in the attachment model $\cT_f(n)$ with $n$ vertices. Note that in both cases, other than the root, the out-degree of each vertex is the same as the degree of the vertex $-1$. This explains the ``$+1$'' in the following well known result for convergence of the density of such vertices proved in \cite{jagers1984growth} for branching processes, and in \cite{rudas2007random} (using CTBP techniques from \cite{jagers1984growth}) for attachment trees. 

\begin{theorem}[\cite{jagers1984growth,rudas2007random}]\label{thm:degree-dist-convg}
Let $\set{p_k:k\geq 1}$ be the pmf as in \eqref{eqn:pk-def}. 	Under Assumption \ref{ass:attach-func}, for each fixed $l\geq 0$, $D(l,t)/|\BP_f(l,t)| \stackrel{a.s.}{\longrightarrow} p_{l+1}$ as $t\to\infty$. The same assertion holds for $D_l(n)/n$ as $n\to\infty$. 
\end{theorem}

\subsection{Proof of Theorem \ref{prop:convg-limit} and Theorem \ref{thm:exponential-mom}: }
\label{sec:proof-exp-mom}

Recall the arrival times $\set{\sigma_i: i\geq 1}$ for the point process $\xi_f$ in \eqref{eqn:xi-f-def}. Define
\begin{equation}
	\label{eqn:y-def}
	Y = \sum_{i=1}^\infty e^{-\lambda \sigma_i}, \qquad Y_m = \sum_{i=m}^{\infty} e^{-\lambda (\sigma_i -\sigma_m)}, \qquad m\geq 0. 
\end{equation}
Note that $Y = Y_0 -1$. Observe that the limit random variable $W_\infty$ of the branching process appearing in Theorem \ref{prop:convg-limit} satisfies the following recursive distributional equation \cite[Equation (3)]{biggins1979continuity} (see also the paragraph following Equation (5) of \cite{biggins1979continuity}):
\begin{equation}
\label{eqn:rde-winf}
	W_\infty \stackrel{d}{=} \sum_{i=1}^{\infty} e^{-\lambda \sigma_i} W_\infty^{\sss(i)},
\end{equation}
where $\set{W_\infty^{\sss(i)}: i\geq 1}$ are a sequence of \emph{i.i.d.} random variables with the same distribution as $W_\infty$ and independent of the point process $\set{\sigma_i:i\geq 1}$. Random variables satisfying such recursive distributional equations have been analyzed extensively (see e.g. \cite{durrett1983fixed,liu2000generalized,liu2001asymptotic,biggins1997seneta,iksanov2004elementary}). In particular, by \cite[Theorem 6]{rosler1992fixed} (or \cite[Theorem 3.4]{alsmeyer2017thin}), finiteness of the moment generating function of $Y$ in a neighborhood of zero guarantees the corresponding assertion for $W_\infty$. The former is achieved in the following Proposition. Note that the law of the random variable $Y$ depends in a non-trivial way on the attachment function $f$ which, in turn, can be very general. The key technical contribution in the proof of the following Proposition is showing that Assumption \ref{ass:lim-sup} on the attachment function is all that is needed to guarantee existence of exponential moments of $Y$. This level of generality is handled by resorting to a careful inductive argument to estimate the moments of $Y$.

\begin{prop}\label{prop:exp-mom-y}
	Under Assumption \ref{ass:lim-sup}, the random variable $Y$ satisfies $\E(e^{tY}) < \infty$ for some $t>0$. 
\end{prop}
\begin{proof}
For $m\geq 0$ and $j\geq m$, write $\sigma_j^{\sss(m)} = \sigma_j - \sigma_m = \sum_{i=m}^{j-1} E_i$. Now for any $k\geq 1$,
\begin{align*}
	\E(Y_m^k) & = \E\bigg(\sum_{m\leq i_i, \ldots, i_k}^\infty e^{-\lambda (\sigma_{i_1}^{\sss(m)} + \cdots +\sigma_{i_k}^{\sss(m)})}\bigg)\\
	&\leq \E\bigg(k \sum_{i=m}^{\infty} \sum_{i\leq j_1, \ldots, j_{k-1}}^{\infty} e^{-\lambda(\sigma_i^{\sss(m)} + \sigma_{j_1}^{\sss(m)} + \cdots + \sigma_{j_{k-1}}^{\sss(m)})}\bigg) \\
	&= k \E\bigg(\sum_{i=m}^{\infty} e^{-\lambda k \sigma_i^{\sss(m)}}\cdot \sum_{i\leq j_1, \ldots, j_{k-1}}^{\infty} e^{-\lambda[(\sigma_{j_1} - \sigma_i)+\cdots +(\sigma_{j_{k-1}} - \sigma_i) ]}\bigg).
\end{align*}
Note that for any $i\geq m$, $\tilde{Y}_i := \sum_{i\leq j_1, \ldots, j_{k-1}}^{\infty} e^{-\lambda[(\sigma_{j_1} - \sigma_i)+\cdots +(\sigma_{j_{k-1}} - \sigma_i) ]} $ is independent of $\sigma_i^{\sss(m)}$ with the same distribution as $Y_i^{k-1}$. Using this and the explicit form of $\sigma_i^{\sss(m)}$ as a sum of independent exponential random variables gives,
\begin{equation}
\label{eqn:ymk-bd}
	\E(Y_m^k) \leq k \sum_{i=m}^{\infty}\E(Y_i^{k-1})\cdot\bigg[ \prod_{j=m}^{i-1}\left(\frac{f(j)}{k\lambda+ f(j)}\right)\bigg].
\end{equation}
Here we use the convention $1 = \prod_{j=m}^{m-1}$. We will use \eqref{eqn:ymk-bd} to obtain moment bounds. By Assumption \ref{ass:lim-sup} there exist $m_0\geq 2$ and $\delta > 0$ such that the attachment function satisfies $f(l) \leq (\lambda -\delta)l $ for all $l\geq m_0$. Define 
\begin{equation}
\label{eqn:mu-cmu-def}
	\mu:= \frac{\lambda}{\lambda -\delta}, \qquad C_\mu:= \frac{e^{(\mu-1)/\mu}}{\mu-1}. 
\end{equation}
The following lemma derives moment bounds for $m\geq m_0$.
\begin{lemma}\label{lem:mom-bd-mzero}
	For any $m\geq m_0$ and $k\geq 1$,
	\begin{equation}
	\label{eqn:mzero-bd}
		\E(Y_m^k) \leq C_\mu^k(k\mu +m)^k. 
	\end{equation}
\end{lemma}

\begin{proof}
\label{pf:lem-mom-bd-mzero}
We will prove the above result via induction on $k$. Precisely, for $k \in \mathbb{N}$,

\begin{quote}
	{\bf Induction hypothesis:} \eqref{eqn:mzero-bd} holds for the given choice of $k$ and all $m\geq m_0$. 
\end{quote} 
First observe that for any fixed $k$,  for $i>m\geq m_0$, 
\begin{align*}
	\log \prod_{j=m}^{i-1} \frac{f(j)}{k\lambda + f(j)} \leq \log \prod_{j=m}^{i-1} \frac{(\lambda - \delta)j}{k\lambda +(\lambda -\delta)j} = \sum_{j=m}^{i-1}\log\left(1-\frac{k\mu}{k\mu+j}\right)\leq -\sum_{j=m}^{i-1} \frac{k\mu}{k\mu+j}.
\end{align*}
Bounding the final term by $-\int_{k\mu+m}^{k\mu+i} (k\mu/z) dz$ shows that for $i\geq m\geq m_0$ and $k\geq 1$, 
\begin{equation}
\label{eqn:prod-bdkmu}
	\prod_{j=m}^{i-1}\frac{f(j)}{k\lambda +f(j)} \leq \left(\frac{k\mu+m}{k\mu+i}\right)^{k\mu}. 
\end{equation}
Further for any $k\geq 1$,
\begin{equation}
\label{eqn:kmu-sum}
	\sum_{i=m}^{\infty} \frac{1}{(k\mu+i)^{k(\mu-1)+1} }\leq \int_{k\mu+m -1}^{\infty} \frac{1}{z^{k(\mu-1)+1}}dz \leq  \frac{C_{\mu}}{k}\frac{1}{(k\mu+m)^{k(\mu-1)}},
\end{equation}
where the last inequality follows upon checking the elementary inequality,
\[\log\left(\frac{k\mu+m}{k\mu+m-1}\right)^{k(\mu-1)} \leq \frac{\mu-1}{\mu}.\]
\noindent {\bf Induction base case ($(k=1)$): } For any $m\geq m_0$,
\begin{equation}
\label{eqn:induc-base}
	\E(Y_m) = \sum_{i=m}^{\infty} \prod_{j=m}^{i-1}\frac{f(j)}{\lambda+f(j)}\leq \sum_{i=m}^{\infty}\left(\frac{\mu+m}{\mu+i}\right)^{\mu}\leq C_\mu(\mu+m),
\end{equation}
where the first bound follows from \eqref{eqn:prod-bdkmu} and the second from \eqref{eqn:kmu-sum}. This verifies the induction hypothesis for $k=1$. 

\noindent {\bf General case:} Now suppose the induction hypothesis is true for $k-1$. By \eqref{eqn:ymk-bd}, \eqref{eqn:prod-bdkmu} and the assumed bound on $\E(Y_m^{(k-1)})$ we get, 
\begin{align}
\label{eqn:induct-general}
	\E(Y_m^k) &\leq k\sum_{i=m}^{\infty} \left(\frac{k\mu+m}{k\mu+i}\right)^{k\mu} C_\mu^{k-1}((k-1)\mu+i)^{k-1}\notag\\
	 &\leq k (k\mu+m)^{k\mu} C_{\mu}^{k-1}\sum_{i=m}^{\infty} \frac{1}{(k\mu+i)^{k(\mu-1)+1}}. 
\end{align}
Now using \eqref{eqn:kmu-sum} on the last sum shows that $\E(Y_m^k)\leq C_\mu^k(k\mu +m)^k$ verifying the induction hypothesis for $k$ and completing the proof of the Lemma. 
\end{proof}

We will now leverage the above bounds to complete the proof. Define,
\begin{equation}
\label{eqn:thet-Lamb-def}
	\theta:=\sup_{1\leq j\leq m_0} \frac{f(j)}{\lambda+f(j)} < 1, \qquad \Lambda_k = \sup_{1\leq i\leq m_0} \E(Y_i^k), \qquad k\geq 1. 
\end{equation}
\begin{lemma}\label{lem:Lamb-bds}
	There exist constants $C_1, C_2 <\infty$ such that for any $k\geq 1$, $\Lambda_k \leq C_1 C_2^k (k+1)!$. 
\end{lemma}

\noindent {\bf Proof of Proposition \ref{prop:exp-mom-y}:} Assuming Lemma \ref{lem:Lamb-bds}, note that this immediately implies that $\forall t\in (0,C_2^{-1})$, $\E(e^{tY_0}) < \infty$. Since $Y_0 = Y+1$, this completes the proof of the proposition. \qed
\\

\noindent {\bf Proof of Lemma \ref{lem:Lamb-bds}:} For $k\geq 1$, using \eqref{eqn:ymk-bd} for $m < m_0$ we get $\E(Y_m^k) \leq S_1(k) + S_2(k)$ where 
\begin{equation}
\label{eqn:s1-def}
	S_1(k) := k\sum_{i=m}^{m_0}\E(Y_i^{k-1})\cdot\bigg[ \prod_{j=m}^{i-1}\left(\frac{f(j)}{k\lambda+ f(j)}\right)\bigg] \leq k \Lambda_{k-1} \sum_{i=m}^{m_0} \theta^{i-m} \leq k D_\theta \Lambda_{k-1},
\end{equation}
where for the rest of this proof,  $D_\theta = \sum_{j=0}^{\infty} \theta^j = 1/(1-\theta)$. Similarly with $S_2(k):=k\sum_{i=m_0+1}^\infty \E(Y_i^{k-1})\cdot\bigg[ \prod_{j=m}^{i-1}\left(\frac{f(j)}{k\lambda+ f(j)}\right)\bigg] $ we have, 
\begin{equation}
\label{eqn:s2-def}
	S_2(k)\leq \theta^{m_0 - m} \cdot k\sum_{i=m_0+1}^\infty \E(Y_i^{k-1})\cdot\bigg[ \prod_{j=m_0}^{i-1}\left(\frac{f(j)}{k\lambda+ f(j)}\right)\bigg] \leq 1\cdot C_\mu^k(k\mu+m_0)^k,
\end{equation} 
where to obtain the final inequality, we have used $\theta <1$ and the same calculations as in \eqref{eqn:induct-general}. Combining \eqref{eqn:s1-def} and \eqref{eqn:s2-def} gives, 
\begin{equation}
\label{eqn:Lamb-recur}
	\Lambda_k\leq k D_\theta \Lambda_{k-1} + C_{\mu}^k(k\mu+m_0)^k, \qquad k\geq 1. 
\end{equation}
Using this bound recursively gives,
\begin{equation}
\label{eqn:Lambda-k-bound}
	\Lambda_k \leq D_\theta^{k-1}k!\Lambda_1 +\sum_{j=0}^{k-1} D_\theta^j C_{\mu}^{k-j}((k-j)\mu + m_0)^{k-j}\frac{k!}{(k-j)!}. 
\end{equation}
Note that for any $j\leq k-1$, $((k-j)\mu+m_0)^{k-j} \leq (\max(\mu, m_0)^{k-j})(k-j+1)^{k-j}$. Further by Stirling's bounds for factorials, 
\begin{equation}
\label{eqn:kfac-bd}
	(k-j+1)^{k-j}\leq \frac{e}{\sqrt{2\pi}} (k-j)! e^{(k-j)}. 
\end{equation}
Let $\tilde{D} = \max\set{D_\theta, eC_\mu\max\set{\mu, m_0}}$. Using \eqref{eqn:kfac-bd} we get
\[C_\mu^k(k\mu + m_0)^k \leq \frac{e}{\sqrt{2\pi}} k! \tilde{D}^k,\]
while for $1\leq j\leq k-1$, 
\[D_\theta^j C_{\mu}^{k-j}((k-j)\mu + m_0)^{k-j}\frac{k!}{(k-j)!} \leq \frac{e}{\sqrt{2\pi}}\tilde{D}^k k!.\]
Using these bounds in \eqref{eqn:Lambda-k-bound} shows that the assertion of Lemma \ref{lem:Lamb-bds} holds with $C_1 = (\Lambda_1 + e/\sqrt{2\pi})$ and $C_2 = \tilde{D}$. 
\end{proof}

{\bf Proof of Theorem \ref{prop:convg-limit}:} Note that Proposition \ref{prop:exp-mom-y} implies $\E(Y^2) <\infty$. Combining this with Assumption \ref{ass:attach-func}, the positivity of $W_{\infty}$ and the $\bL^2$ convergence assertion in \eqref{eqn:bp-convg-winf} follow from \cite[Corollary 4.2 and Theorem 4.3]{jagers1984growth} while the a.s. convergence follows from \cite[Theorem 5.4]{nerman1981convergence}. The assertion on the continuity and support of the distribution now follows from \cite[Theorem 1 and Theorem 2]{biggins1979continuity}. \qed

{\bf Proof of Theorem \ref{thm:exponential-mom}:} Recall that the limit random variable $W_\infty$ of the branching process satisfies the recursive distributional equation \eqref{eqn:rde-winf}. The finiteness of the moment generating function of $W_\infty$ claimed in Theorem \ref{thm:exponential-mom} now follows from \cite[Theorem 6]{rosler1992fixed} (or \cite[Theorem 3.4]{alsmeyer2017thin}) upon using Proposition \ref{prop:exp-mom-y}, which establishes the finiteness of the moment generating function of $Y$ in a neighborhood of the origin. \qed

\subsection{Proof of Theorem \ref{thm:expon-convergence} and Corollary \ref{cor:sup-bound}: }

\subsubsection{Proof of Theorem \ref{thm:expon-convergence}:} 
Recall from Section \ref{sec:bp-prelim} that for fixed $l\geq 0$,  $D(l,t)$ denotes the number of individuals with $l$ children in $\BP_f(t)$. Let $\cD(l,t)$ denote the corresponding set of vertices. Recall the asymptotics for these objects encapsulated in Theorem \ref{thm:degree-dist-convg}. We will use $\set{\BP_f(t):t\geq 0}$ to also denote the natural filtration of the process. Note that conditional on $\BP_f(t)$, for a vertex $v\in \cD(l,t)$, the point process describing the new offspring produced after time $t$ has distribution $\xi_f^{\sss(l)}$ as in \eqref{eqn:xi-k-f-t-def} (with $\mu_f^{\sss(l)}$ the corresponding mean measure) whilst all ensuing new children reproduce according to the original point process $\xi_f$ (all of these point processes conditionally independent given $\BP_f(t)$). Thus, recalling $Z_f(t):=|\BP_f(t)|$ and $m_f(t) = \E(Z_f(t))$, it is easy to see that for any $u\geq 0$,
\begin{equation}
\label{eqn:543}
	\E(Z_f(t+u)|\BP_f(t)) = \sum_{l=0}^{\infty} D(l,t)\left(1+\int_0^u m_f(u-s)\mu^{\sss(l)}(ds)\right):= \sum_{l=0}^{\infty} D(l,t) \lambda_l(u).
\end{equation}
The analysis hinges on the following technical tools derived in \cite{banerjee2018fluctuation}.
\begin{lemma}\label{lem:cp-techn-tools}
	Consider a branching process $\BP_f(\cdot)$ with attachment function $f$. Then we have the following:
	\begin{enumeratea}
		\item \label{it:techn-a} If $f$ satisfies Assumptions \ref{ass:attach-func} and \ref{ass:lim-sup}, then by \cite[Lemma 9.2]{banerjee2018fluctuation}, with $\set{p_l:\geq 0}$ as in \eqref{eqn:pk-def}, $\set{\lambda_l(\cdot):l\geq 0}$ as in \eqref{eqn:543} and $W_\infty$ as in Theorem \ref{prop:convg-limit}, there exist constants  $\eps, C_1, C_2$ such that for all $t \ge 0$, 
		\[\E\left(\sup_{u\in [0,\eps t]}\left|e^{-\lambda t} \E(Z_f(t+u)|\BP_f(t)) - \sum_{l=0}^\infty \lambda_l(u) p_{l+1} W_\infty\right|\right)\leq C_1e^{-C_2 t}.\]
		\item \label{it:techn-b} If $f$ satisfies Assumption \ref{ass:attach-func}, then by \cite[Lemma 7.11]{banerjee2018fluctuation}, there exist constants $C_3, \gamma >0$ and $0<\omega<1$ such that for all $t \ge 0$, 
		\begin{multline*}
		\pr\left[\left.\sup_{u\in [0,1]}\left|e^{-\lambda t} Z_f(t+u) - e^{-\lambda t} \E(Z_f(t+u)|\BP_f(t))\right| > e^{-\lambda t}(Z_f(t))^{\omega}\right|\BP_f(t)\right]\\
		 \leq \frac{C_3}{(Z_f(t))^\gamma}.
		\end{multline*}
	\end{enumeratea} 
\end{lemma}

\begin{rem}
	In \cite{banerjee2018fluctuation}, part (a) above was shown to hold under the stronger condition that $\lim_{i\to\infty} f(i)/i\to C$. However, it was shown that this stronger condition implied the Assumption \ref{ass:lim-sup} which was then used in the proof without further reference to existence of the limit. Moreover, the finiteness of the second moment of $Y = \sum_{i=1}^\infty e^{-\lambda \sigma_i}$ needed for the proof of part (a) in \cite{banerjee2018fluctuation} is a direct consequence of Proposition \ref{prop:exp-mom-y}.
\end{rem}

Now write $A(t,u):=|e^{-\lambda t} Z_f(t+u) - e^{-\lambda t} \E(Z_f(t+u)|\BP_f(t))| $ for the object in Lemma \ref{lem:cp-techn-tools} \eqref{it:techn-b}. Note that for any $\delta >0$, 
\begin{equation}
\label{eqn:615}
	\pr(\sup_{u\in [0,1]} A(t,u) > e^{-\delta t}) \leq \pr(\sup_{u\in [0,1]} A(t,u) > e^{-\delta t}, Z_f(t)\leq e^{(\lambda+\delta})t) + \pr(Z_f(t)> e^{(\lambda + \delta)t}). 
\end{equation}
By Proposition \ref{prop:exp-mom-y}, $C_4 := \sup_{t\geq 0} e^{-\lambda s} \E(Z_f(s)) < \infty$. Thus by Markov's inequality,  for any $\delta >0$, $t \ge 0$,
\begin{equation}
\label{eqn:623}
	\pr(Z_f(t) > e^{(\lambda+\delta)t}) \leq C_4 e^{-\delta t}. 
\end{equation}
Now we deal with the first term in \eqref{eqn:615}. With $\omega$ as in Lemma \ref{lem:cp-techn-tools} \eqref{it:techn-b}, take $\delta \in (0, C_2/2)$ such that $\lambda - (\lambda +\delta)\omega > \delta$. With this choice of $\delta$, check that,  \[\set{\sup_{u\in [0,1] } A(t,u) > e^{-\delta t} , Z_f(t) \leq e^{(\lambda +\delta) t}} \subseteq \set{\sup_{u\in [0,1] } A(t,u) > e^{-\lambda t} (Z_f(t))^\omega , Z_f(t) \leq e^{(\lambda +\delta) t}}.\]
Thus Lemma \ref{lem:cp-techn-tools} \eqref{it:techn-b} (for the first inequality) and the second bound in Lemma \ref{lem:zft-lbd-1} imply that,
\begin{equation}
\label{eqn:738}
	\pr\left(\sup_{u\in [0,1]} A(t,u) > e^{-\delta t}, Z_f(t)\leq e^{(\lambda+\delta)}t\right) \leq \E\left[\frac{C_3}{(Z_f(t))^\gamma}\right]\leq C_5 e^{-C_6 t}.  
\end{equation}
Using \eqref{eqn:623} and \eqref{eqn:738} in \eqref{eqn:615}, we obtain
\begin{equation}
\pr\left(\sup_{u\in [0,1]}|e^{-\lambda t} Z_f(t+u) - e^{-\lambda t} \E(Z_f(t+u)|\BP_f(t))| > e^{-\delta t}\right) \le C_7 e^{-C_8 t}.
\end{equation}
This bound, combined with Markov's inequality and Lemma \ref{lem:cp-techn-tools} \eqref{it:techn-a} (recalling that $\delta \in (0, C_2/2)$), shows that for all $t \ge 0$,
\begin{equation}
\label{eqn:814}
	\pr\left[\sup_{u\in [0,1]} \left|e^{-\lambda t} Z_f(t+u) - W_\infty\sum_{l=0}^\infty \lambda_l(u) p_{l+1}\right| > e^{-\delta t}\right] \leq C_9 e^{-C_{10} t}. 
\end{equation}
Before proceeding we need the following identity. 
\begin{lemma}\label{lem:sumlap-1}
	For any $u\geq 0$, $\sum_{l=0}^\infty e^{-\lambda u} \lambda_l(u) p_{l+1} =1$.
\end{lemma}
\begin{proof}
	By Theorem \ref{prop:convg-limit}, for any $u \ge 0$, $e^{-\lambda(t+u)} Z_f(t+u) \convas W_\infty$ as $t \rightarrow \infty$, with $W_\infty > 0$ a.s. Using \eqref{eqn:814} now concludes the proof. 
	
\end{proof}

Using Lemma \ref{lem:sumlap-1} in \eqref{eqn:814} now shows that for all $t \ge 0$, 
\begin{equation}
\label{eqn:828}
	\pr\left[\sup_{u\in [0,1]} \left|e^{-\lambda (t+u)} Z_f(t+u) - W_\infty\right| > e^{-\delta t}\right] \leq C_9 e^{-C_{10} t}. 
\end{equation}
 By the union bound, for any $t \ge 0$, $\pr(\sup_{s\geq  t} |e^{-\lambda s} Z_f(s) - W_\infty| > e^{-\delta t})$ can be bounded by
 \[\sum_{k=0}^\infty\pr\left[\sup_{u\in [0,1]} \left|e^{-\lambda (t+k+u)} Z_f(t+k+u) - W_\infty\right| > e^{-\delta (t+k)}\right] \leq \sum_{k=0}^\infty C_9 e^{-C_{10} (t+k)} \leq C_{11} e^{-C_{10} t}, \]
where the second inequality follows from \eqref{eqn:828}. This concludes the proof of Theorem \ref{thm:expon-convergence}. \qed

\subsubsection{Proof of Corollary \ref{cor:sup-bound}:} Note that for any fixed $T> 0$,

\begin{align}
	\pr\left(\sup_{s > T} e^{-\lambda s} Z_f(s) \geq A\right) &\leq \pr\left(\sup_{s> T}|e^{-\lambda s}Z_f(s) - W_\infty| \geq A/2\right) + \pr(W_\infty \geq A/2) \notag\\
	&\leq C_1e^{-C^\prime T} +  C_2/A^2, \label{eqn:1259}
\end{align}
where the first bound in \eqref{eqn:1259} follows from Theorem \ref{thm:expon-convergence} and the second bound follows from Chebyshev's inequality and Theorem \ref{thm:exponential-mom} which in particular implies $\E(W_\infty^2) <\infty$. Now we consider the interval $[0,T]$. Divide the interval into a partition $0 = t_0 < t_1 < \cdots < t_N = T$ of mesh size one (other than the last interval which could be of length $\leq 1$). Since $Z_f(s)$ is $\uparrow$ in $s$ a.s, we have for any $i$, $\pr\left(\sup_{s\in [t_i, t_{i+1}]} e^{-\lambda s} Z_f(s) \geq A\right) \leq \pr\left(e^{-\lambda t_i}Z_f(t_{i+1}) \geq A\right)$. This can be bounded as,
\begin{equation}
\label{eqn:1315}
	\pr\left(e^{-\lambda t_{i+1}}Z_f(t_{i+1}) \geq Ae^{-\lambda(t_{i+1} - t_i)}\right) \leq C_3/ A^2,
\end{equation}
where for the last bound we have used the fact that Theorem \ref{prop:convg-limit} implies that $\sup_t \E([e^{-\lambda t}Z_f(t)]^2) < \infty$. Using \eqref{eqn:1259} and \eqref{eqn:1315} and the union bound now gives, 
\[\pr\left(\sup_{s>0}e^{-\lambda s}Z_f(s)\geq A\right)\leq C_1e^{-C^\prime T} + C_4(T+1)/A^2.\]
Taking $T = M\log{A}$ for sufficiently large $M$ completes the proof. \qed

\section{Proofs: Root finding algorithms}

\subsection{Deterministic properties of Jordan centrality}
\label{sec:proofs-det-jordan}
The following lemma collects useful properties of the Jordan centrality measure $\Psi$ defined in \eqref{eqn:psi-def}. We will only use property (a) in this paper, but list the other properties as they might be useful for readers unfamiliar with this centrality measure.

\begin{lemma}\label{lem:jord-cent-det}
	Let $\vt \in \bT$ be a tree on $n\geq 3$ vertices and let $v^* = v^*(\vt)$ denote a centroid of $\vt$. The following properties hold for the centroid and the Jordan centrality measure:
	\begin{enumeratea}
		\item \label{it:jord-nesc-suff} By \cite[Lemma 13]{jog2016analysis}, for any two vertices $u,v\in \vt$, 
		\[\Psi_{\vt}(u) \leq \Psi_{\vt}(v) \Longleftrightarrow |\desc{(\vt,v)}{u}|\geq |\desc{(\vt,u)}{v}|. \]
		\item By \cite[Lemma 2.1]{jog2018persistence} $\Psi_{\vt}(v^*)\leq n/2$. Further $\vt$ can have at most two centroids which have to be adjacent to each other. If two distinct centroids $u^*$ and $v^*$ exist then 
		\[\Psi_{\vt}(u^*) = |\desc{(\vt,u^*)}{v^*}|, \qquad \Psi_{\vt}(v^*) = |\desc{(\vt,v^*)}{u^*}|. \]
		\item Consider a sequence of growing trees $\set{\cT(n):n\geq 1}$ with vertex set $V(\cT(n)) = \set{v_1, v_2, \ldots, v_n}$ in order of appearance of the vertices. Let $v^*(n)$ denote a centroid in $\cT(n)$. Then by \cite[Lemma 2.3]{jog2018persistence} $|(\desc{\cT_{n+1},v_{n+1})}{v^*(n)}|\geq n/2$. 
	\end{enumeratea} 
\end{lemma}

\subsection{Proof of Theorem \ref{thm:upp-root-find}:}

We will need a few preliminary results before commencing with the proof of the theorem. We first setup some notation. Fix an attachment function $f$ (and recall $f_* =\inf_{i\geq 1} f(i)$). Further throughout this section we will assume $f(i) \leq \overline{C}_f \cdot i + \beta$ for all $i$.  Consider the following two associated branching process: 

\begin{enumeratea}
	\item Let $\set{\LBP_f(t): t\geq 0}$ be the branching process where the root reproduces at constant rate $f_*$ whilst all other individuals produce offspring using the usual offspring point process $\xi_f$ as in \eqref{eqn:xi-f-def}. We will refer to this as the {\bf lower bounding} branching process (owing to the next lemma).  Write $\set{\LZF(t):t\geq 0}$ for the corresponding population size. 
	\item Fix $i\geq 0$ and let $\set{\BP_f^{\sss(i)}(t):t\geq 0}$ be the branching process where the root uses the point process $\xi_f^{\sss(i)}$ as in \eqref{eqn:xi-k-f-t-def} whilst all other individuals produce offspring using the usual offspring point process $\xi_f$ as in \eqref{eqn:xi-f-def}. Note that $\BP_f^{\sss(0)} = \BP_f$.

\end{enumeratea}

\begin{lemma}\label{lem:lower-bd}
	Suppose $f$ satisfies Assumptions \ref{ass:attach-func} and \ref{ass:lim-sup}. 
	\begin{enumeratea}
		\item For any $i\geq 0$ one can couple $\LBP_f$ and $\BP_f^{\sss(i)}$ so that $|\BP_f^{\sss(i)}(t)| \geq \LZF(t)$ for all $t\geq 0$. 
		\item Further writing $\LW := \liminf_{t\to\infty} e^{-\lambda t} \LZF(t)$ we have  $\LW >0$ a.s. and $\exists \kappa > 0$ such that for all $s\in (0,\kappa)$, $\E(\exp(s\LW)) < \infty$. 
	\end{enumeratea}
\end{lemma}

\begin{proof}
	Part(a) follows easily by comparing the rates of evolution of individual vertices in each of the two processes. To prove (b) first note that by (a), almost surely, $\liminf_{t\to\infty} e^{-\lambda t} \LZF(t) \stod \lim_{t\to\infty} e^{-\lambda t} |\BP_f(t)|:=W_\infty$ as in Theorem \ref{prop:convg-limit}. Thus the second set of assertions about finite exponential moments follow from Theorem \ref{thm:exponential-mom}. To prove the first assertion,  let $\underline{\sigma}_1$ denote the time of birth of the first individual of the root in $\LBP_f$ and for any $s\geq 0$ let $Z_{f, 1}(s)$ denote the size of the subtree of the first child of the root $s$ units of time after birth of this child. Now for any $t\geq \underline{\sigma}_1$, one has $e^{-\lambda t} \LZF(t) \geq e^{-\lambda \underline{\sigma}_1} e^{-\lambda(t-\underline{\sigma}_1)} Z_{f, 1}(t-\underline{\sigma}_1)$. Since $e^{-\lambda s} Z_{f,1}(s) \convas W_{1,\infty}$ as $s \rightarrow \infty$, where $W_{1,\infty}  \stackrel{d}{=} W_{\infty}$ as in Theorem \ref{prop:convg-limit}, the first assertion on strict positivity follows from Theorem \ref{prop:convg-limit}. 
	
\end{proof}

Now recall that in the continuous time construction of $\cT_f(\cdot)$ in Lemma \ref{lem:ctb-embedding-no-cp}, the process evolves by starting two independent branching processes $\BP_f^{\sss(i \downarrow)}$ from the two initializers $v_1$ and $v_2$ which are connected by an edge. Define the stopping times
\begin{equation}
\label{eqn:eta-def}
\eta^{\sss(i)}_j = \inf\set{t\geq 0: |\BP_f^{\sss(i \downarrow)}(t)| =j}, \ j \ge 1, \ i=1,2. 
\end{equation}
Thus the stopping times $\{\eta^{\sss(1)}_j : j \ge 1\}$ track when the tree below $v_1$ becomes of size $j$. The next lemma derives bounds showing that the tree below $v_2$ cannot be ``too small'' relative to the size of the tree below $v_1$ at \emph{any} of these times.  We need the following constants, 
\begin{equation}
\label{eqn:201}
	\tilde{C} = \frac{1}{2(2\overline{C}_f + \beta)^2} \sum_{l=1}^\infty \frac{1}{l^2}, \qquad B = \frac{4\tilde{C}f_*^3}{2\overline{C}_f + \beta}. 
\end{equation}
\begin{lemma}\label{lem:no-too-small}
	There exists a constant $D>0$ such that for any $M> 0$ one can obtain integer $\tilde K_M \ge 2$ such that, with B as in \eqref{eqn:201},
	\[\pr\left(\cup_{j=k}^\infty \set{|\BP_f^{\sss(2 \downarrow)}(\eta^{\sss(1)}_j)| \leq M\log{j}} \right) \leq \frac{De^{2\sqrt{(B+M)\log{k}}}}{k^{f_*/(2\overline{C}_f+\beta)}}, \ \ k \ge \tilde K_M.\]	
\end{lemma}

\begin{rem}
	Owing to the square root in the exponent appearing in the above bound, the bound is small for large $k$ (for fixed choice of $M$). 
	Also note that, by symmetry, the same bound holds if we replace $\BP^{\sss(2\downarrow)}_f$ by $\BP_f^{\sss(1 \downarrow)}$ and $\eta^{\sss(1)}_j$ by $\eta^{\sss(2)}_j$ in the above lemma.
\end{rem}

\begin{proof}
	To ease notation, in this proof we will write $\eta_j$ for $\eta_j^{(1)}$. Recall the upper bounding $\LPA_{\overline{C}_f, \beta}$ process for $\BP^{\sss(1 \downarrow)}$ as in Proposition \ref{prop:bpf1-2pdom}(b). Define 
	\begin{equation}
	\label{eqn:eta-tild-def}
		\widetilde{\eta}_j:= \inf\set{t\geq 0:|\LPA_{\overline{C}_f, \beta}(t)| = j}.
	\end{equation}
	Obviously $\eta_j \geq \widetilde{\eta}_j$. Further it is easy to check that 
	\begin{equation}
	\label{eqn:tild-etj-dist}
		\widetilde{\eta}_j \stackrel{d}{=}  \sum_{l=1}^{j-1}\frac{Y_l}{(2\overline{C}_f+\beta)\cdot l-\overline{C}_f},
	\end{equation}
	where $\set{Y_l:l\geq 1}$ are a sequence of \emph{i.i.d.} exponential rate one random variables. Now we derive a concentration bound for $\widetilde{\eta}_k$ for $k \ge 2$ using standard moment generating function techniques. Note that for any $k \ge 2$, $\theta>0$,
\begin{align*}
\log \mathbb{E}\left(e^{-\theta \widetilde{\eta}_k}\right) &\le -\sum_{l=1}^{k-1}\log\left(1 + \frac{\theta}{(2\overline{C}_f + \beta) l}\right)\\
&\le -\sum_{l=1}^{k-1} \frac{\theta}{(2\overline{C}_f + \beta) l} + \frac{\theta^2}{2(2\overline{C}_f + \beta)^2}\sum_{l=1}^{\infty}\frac{1}{l^2} \le -\frac{\theta}{(2\overline{C}_f + \beta)} \log k + \theta^2 \tilde{C},
\end{align*}
with $\tilde{C}$ as in \eqref{eqn:201}. Therefore, for any $t >0$,
$$
\pr\left(\widetilde{\eta_k} \leq \frac{\log{k}}{2\overline{C}_f+\beta} - t\right) \le e^{\theta \left(\frac{\log{k}}{2\overline{C}_f+\beta}\log k - t\right)}\mathbb{E}\left(e^{-\theta \widetilde{\eta}_k}\right) \le e^{\theta^2 \tilde{C} - \theta t}.
$$
Optimizing over $\theta$, we obtain for any $t>0$ and $k\geq 2$, 
	\begin{equation}
	\label{eqn:eta-tail}
		\pr\left({\eta_k} \leq \frac{\log{k}}{2\overline{C}_f+\beta} - t \right) \leq \pr\left(\widetilde{\eta_k} \leq \frac{\log{k}}{2\overline{C}_f+\beta} - t \right) \leq  e^{-t^2/4\tilde{C}}.
	\end{equation}
	For fixed $M$ and let 
	\begin{equation}
	\label{eqn:m-prime}
		M^\prime:= \frac{M}{f_*^2} + \frac{4\tilde{C}f_*}{2\overline{C}_f+\beta}.
	\end{equation}
	Moving to the evolution $\BP^{\sss(2 \downarrow)}_f(\cdot)$, again Proposition \ref{prop:bpf1-2pdom}(b), note that $\BP^{\sss(2 \downarrow)}_f$ can be \emph{lower} bounded by $\yu_{f_*}$. Thus for fixed $k$, 
	\begin{align*}
		\pr(|\BP^{\sss(2 \downarrow)}_f(\eta_k)| \leq 2M \log{k}) \leq & \pr\bigg(\eta_k \leq \frac{\log{k}}{2\overline{C}_f+\beta} - \sqrt{M^\prime \log{k}}\bigg)\\ 
		& + \pr\left(\left|\yu_{f_*}\bigg(\frac{\log{k}}{2\overline{C}_f+\beta} - \sqrt{M^\prime \log{k}}\bigg)\right| \leq 2M\log{k}\right).
	\end{align*}
	Using \eqref{eqn:eta-tail} with $t = \sqrt{M^\prime \log{k}}$ to estimate the first term, and the explicit distribution of the Yule process from Lemma \ref{lem:yule-prop} along with the union bound to estimate the second term, we obtain after some algebra, 
	\begin{equation}
	\label{eqn:250}
		\pr(|\BP^{\sss(2 \downarrow)}_f(\eta_k)| \leq 2M \log{k}) \leq D^{\prime} \frac{ e^{2\sqrt{(B+M)\log{k}}}}{k^{f_*/(2\overline{C}_f+\beta)}},
	\end{equation}
	where $D^{\prime} = 1+\sup_{x\geq 0} 2x^2 e^{-x}$ and $B$ is as in \eqref{eqn:201}. Now for any $m\geq 1$, note that 
	\begin{align*}
	\cup_{j=k^{m}}^{k^{m+1}}\set{|\BP^{\sss(2 \downarrow)}_f(\eta_j)| \leq M \log{j}} &\subseteq \set{|\BP^{\sss(2 \downarrow)}_f(\eta_{k^m})| \leq  M \log{k^{m+1}}}\\
	&\subseteq \set{|\BP^{\sss(2 \downarrow)}_f(\eta_{k^m})| \leq  2M \log{k^{m}}}.
	\end{align*}
	 Using this observation, along with \eqref{eqn:250}, we get 
	\begin{align}
	\label{eqn:255}
		\pr\left(\cup_{j=k^{m}}^{k^{m+1}}\set{|\BP^{\sss(2 \downarrow)}_f(\eta_j)| \leq M \log{j}}\right) &\leq \pr\left(|\BP^{\sss(2 \downarrow)}_f(\eta_{k^m})| \leq  2M \log{k^{m}}\right)\notag\\
		& \leq D^{\prime} \frac{ e^{2\sqrt{(B+M) m\log{k}}}}{k^{m f_*/(2\overline{C}_f+\beta)}}. 
	\end{align}
	Let $\tilde K_M \ge 2$ be such that $\frac{ e^{2\sqrt{(B+M) \log{k}}}}{k^{f_*/(2\overline{C}_f+\beta)}} \le \frac{1}{2}$ for all $k \ge \tilde K_M$. Then, from the above bound, summing over $m \ge 1$,
	\begin{align*}
	\pr\left(\cup_{j=k}^\infty \set{|\BP_f^{\sss(2 \downarrow)}(\eta^{\sss(1)}_j)| \leq M\log{j}} \right) &\le \sum_{m=1}^{\infty}D^{\prime} \frac{ e^{2\sqrt{(B+M) m\log{k}}}}{k^{m f_*/(2\overline{C}_f+\beta)}}\\
	& \le \sum_{m=1}^{\infty}D^{\prime} \left(\frac{ e^{2\sqrt{(B+M) \log{k}}}}{k^{ f_*/(2\overline{C}_f+\beta)}}\right)^m \le 2D^{\prime} \frac{e^{2\sqrt{(B+M) \log{k}}}}{k^{f_*/(2\overline{C}_f+\beta)}},
	\end{align*}
	for all $k \ge \tilde K_M$. This completes the proof.
\end{proof}

\noindent {\bf Proof of Theorem \ref{thm:upp-root-find}:} We start with (a). We will continue to use the notation and constructs setup for the proof of Lemma \ref{lem:no-too-small}. We will assume $\set{\cT_f(n):n\geq 2}$ has been constructed using the continuous time embedding as in Lemma \ref{lem:ctb-embedding-no-cp} so that in particular $\cT_f(n) = \widetilde{\BP}_f(\eta_n^{\sss(1,2)})$ for the stopping times $\eta_n^{\sss(1,2)}$ defined in \eqref{eqn:tm-stop-def}. As under Definition \ref{def:root-find}, for $K\geq 2$, let $H_{K,\Psi}(\cT^\circ(n))$ denote the $K$ vertices with the smallest $\Psi$ values.   Recall the stopping times $\{\eta^{\sss(i)}_j : j \ge 1, i = 1,2\}$ defined in \eqref{eqn:eta-def}. 
For fixed $M> 0$, define  
\[\cS_K^{M}:= \left(\cup_{j=\lfloor K/2 \rfloor + 1}^{\infty}\set{|\BP_f^{\sss(2 \downarrow)}(\eta_j^{\sss(1)})| \leq M\log{j}}\right) \bigcup \left(\cup_{j=\lfloor K/2 \rfloor + 1}^{\infty}\set{|\BP_f^{\sss(1 \downarrow)}(\eta_j^{\sss(2)})| \leq M\log{j}}\right) . \]
In words, this is the event when there exists a $j \ge K$ such that one of the two subtrees starting from the two initiators $v_1$ and $v_2$ reaches size $j$ but the other is still only logarithmic in this size. The following notation will be convenient later.  For fixed $M$ as above, let 
\begin{equation}
\label{eqn:km-def}
	K_M =  \tilde K_M \vee \inf\set{i\geq 4: l/2 > M\log{l/2} \ \forall \ l\geq i}.
\end{equation}
 Also note that for any $i\geq 3$,
 \begin{equation}
 \label{eqn:846}
 	|\desc{(\cT^\circ(i), v_i)}{v_1}| \geq \min\set{|\BP_f^{\sss(1 \downarrow)}(\eta_i^{\sss(1,2)})|, |\BP_f^{\sss(2 \downarrow)}(\eta_i^{\sss(1,2)})|}. 
 \end{equation}
 
 \begin{lemma}\label{lem:909}
 	Fix $K\geq K_M$ and consider the event $(\cS_K^M)^c$. On this event, for all $i> K$, $|\desc{(\cT^\circ(i), v_i)}{v_1}| \geq M\log(i/2)$. 
 \end{lemma}
 \begin{proof}
    For any $i\geq 2$, either $|\BP_f^{\sss(1 \downarrow)}(\eta_i^{\sss(1,2)})| \geq i/2$ or $|\BP_f^{\sss(2 \downarrow)}(\eta_i^{\sss(1,2)})| \geq i/2$. Fix $K\geq K_M$ and consider the event $(\cS_K^M)^c$. On this event, for fixed $i> K$, if $|\BP_f^{\sss(1 \downarrow)}(\eta_i^{\sss(1,2)})| \geq i/2$ then $|\BP_f^{\sss(2 \downarrow)}(\eta_i^{\sss(1,2)})| \geq M\log(i/2)$ and vice-versa. Using \eqref{eqn:846} and the definition of $K_M$ in \eqref{eqn:km-def} completes the proof. 

 \end{proof}

 Note that $\set{v_1 \notin H_{K,\Psi}(\cT^\circ(n))} \subseteq \set{\exists i> K: \Psi_{\cT^\circ(n)}(v_i) \le\Psi_{\cT^\circ(n)}(v_1) }$. Further by Lemma \ref{lem:jord-cent-det}(a)
\[\set{\exists i> K: \Psi_{\cT^\circ(n)}(v_i) \le \Psi_{\cT^\circ(n)}(v_1) } = \set{\exists i> K: |\desc{(\cT^\circ(n),v_1)}{v_i}| \ge |\desc{(\cT^\circ(n),v_i)}{v_1}|}. \]
Thus by the union bound and Lemma \ref{lem:909},  for any $K\ge K_M$, 
\begin{align}
	&\pr({v_1 \notin H_{K,\Psi}(\cT^\circ(n))})\notag\\
	&\leq \pr\left(\set{\exists i> K: |\desc{(\cT^\circ(n),v_1)}{v_i}| \ge |\desc{(\cT^\circ(n),v_i)}{v_1}|} \cap (\cS_K^{M})^c \right) + \pr(\cS_K^M), \notag\\
	&\leq \pr(\cS_K^M)\notag\\
	&\qquad + \sum_{i=K+1}^n \pr(|\desc{(\cT^\circ(n),v_1)}{v_i}| \ge |\desc{(\cT^\circ(n),v_i)}{v_1}|, |\desc{(\cT^\circ(i),v_i)}{v_1}| \ge M\log{i/2}).  \label{eqn:945}
\end{align}
Consider any $M>0$. Fix $K \ge K_M$ and $i\geq K$. We will analyze the contribution of the $i$-th term in the summand above. 
\begin{lemma}\label{lem:term-i}
	Consider the setting above. For any $M>0$, fix $K \ge K_M$ and $i\geq K$.  There exist constants $C_1^{\star}, C_2^{\star} >0$ independent of $M$ such that 
	\[\limsup_{n\to\infty} \pr\bigg(|\desc{(\cT^\circ(n),v_1)}{v_i}| \ge |\desc{(\cT^\circ(n),v_i)}{v_1}|, |\desc{(\cT^\circ(i),v_i)}{v_1}| \ge M\log{i/2}\bigg) \leq \frac{C_1^{\star}}{(i/2)^{C_2^{\star} M}}. \]
\end{lemma}
\begin{proof}
	
We will switch back to continuous time to understand the evolution of $\desc{(\cT^\circ(\cdot),v_1)}{v_i}$ and $\desc{(\cT^\circ(\cdot),v_i)}{v_1}$. Note that, conditional on $\cT^\circ(i) = \widetilde{\BP}_f(\eta_i^{\sss(1,2)})$, these two processes evolve (conditionally) independently as follows: 

\begin{enumerate}
	\item the process $\desc{(\cT^\circ(\cdot),v_1)}{v_i}$ evolves in continuous time as a continuous time branching process $\BP_f$ originating at vertex $v_i$ as in Def. \ref{def:ctbp}. We will denote this process by $\set{\cT^{\sss(i)}(\cdot):t\geq 0}$, with $\cT^{\sss(i)}(0) = \desc{(\cT^\circ(i),v_1)}{v_i}$ (tree with only one vertex $v_i$). Write for $t\geq 0$, $\cY^{\sss(1)}(t) := |\cT^{\sss(i)}(t)|$. 
	\item the process $\desc{(\cT^\circ(\cdot),v_i)}{v_1}$ evolves in continuous time as follows: for each vertex $v \in \desc{(\cT^\circ(i),v_i)}{v_1} $ start a branching process where if vertex $v$ has degree $k\geq 1$ in $\desc{(\cT^\circ(i),v_i)}{v_1}$ then the root of this branching process has offspring point process $\xi_f^{\sss(k)}$ as in \eqref{eqn:xi-k-f-t-def} whilst all ensuing individuals have the original offspring point process $\xi_f$. In the notation of Lemma \ref{lem:lower-bd}, we start a branching process with distribution $\BP^{\sss(k)}_f$ below vertex $v$.  We will let $\cT^{i\to 1}(\cdot)$ denote this (forest valued) process comprising branching processes originating from each vertex in $\desc{(\cT^\circ(i),v_i)}{v_1}$. For $t\geq 0$, let $\cY^{\sss(2)}(t) := |\cT^{i\to 1}(t)|$ denote the size of this process. 
\end{enumerate} 
Now recall the lower bounding process from Lemma \ref{lem:lower-bd} and let $\set{\LZF^{\sss(j)}(\cdot):j\geq 1}$ be a sequence of \emph{i.i.d.} processes with distribution $\LZF(\cdot)$ as in the lemma. On the set $\set{|\desc{(\cT^\circ(i),v_i)}{v_1}| \ge M\log{i/2}}$, 
\begin{equation}
\label{eqn:1013}
	\sum_{j=1}^{\lfloor M\log(i/2) \rfloor}\LZF^{\sss(j)}(\cdot)\stod \cY^{\sss(2)}(\cdot).  
\end{equation}
Assuming these processes are constructed on the same probability space as $\set{\cY^{\sss(j)}(\cdot): j=1,2}$, we get for fixed $i$, 
\begin{align}
	\limsup_{n\to\infty}\pr\bigg(|\desc{(\cT^\circ(n),v_1)}{v_i}| &\ge |\desc{(\cT^\circ(n),v_i)}{v_1}|, |\desc{(\cT^\circ(i),v_i)}{v_1}| \ge M\log{i/2}\bigg)\notag\\
	&\leq \limsup_{t\to\infty}\pr\bigg(e^{-\lambda t} \cY^{\sss(1)}(t) \ge e^{-\lambda t} \sum_{j=1}^{\lfloor M\log(i/2) \rfloor} \LZF^{\sss(j)}(t)\bigg)\notag\\
 &\leq \pr\bigg( \limsup_{t\to\infty} e^{-\lambda t} \cY^{\sss(1)}(t) \geq \liminf_{t\to\infty} e^{-\lambda t} \sum_{j=1}^{\lfloor M\log(i/2) \rfloor} \LZF^{\sss(j)}(t)\bigg) \notag\\
 &\leq \pr\bigg(W_\infty \geq \sum_{j=1}^{\lfloor M\log{(i/2) \rfloor}} \LW^{\sss(j)}\bigg), \label{eqn:1035}
\end{align}
where $W_\infty$ has distribution as in Theorem \ref{prop:convg-limit} and independent of the \emph{i.i.d} sequence $\set{\LW^{\sss(j)}:j\geq 1}$, which have distribution $\LW$ as in Lemma \ref{lem:lower-bd}. The second inequality above follows by Fatou's Lemma.

By Lemma \ref{lem:lower-bd}, $\LW$ is a sub-exponential random variable with $\E(\LW) = \mu^* > 0$. Thus, there exist $\tilde{C}_1, \tilde{C}_2 > 0$ (independent of $M$) such that, 
\begin{equation}
\label{eqn:1047}
	\pr\bigg(\sum_{j=1}^{\lfloor M\log(i/2) \rfloor} \LW^{\sss(j)} \leq \frac{\mu^*}{2} M \log{(i/2)}\bigg) \leq \frac{\tilde{C}_1}{(i/2)^{\tilde{C}_2 M}}. 
\end{equation}
Similarly since $W_\infty$ has exponential tails by Theorem \ref{thm:exponential-mom}, there exist $C_1^\prime, C_2^\prime$  (independent of $M$) such that 
\begin{equation}
\label{eqn:1052}
	\pr\bigg(W_\infty \geq \frac{\mu^*}{2} M\log{i/2}\bigg) \leq \frac{C_1^\prime}{(i/2)^{C_2^\prime M}}.
\end{equation}
Combining \eqref{eqn:1047} and \eqref{eqn:1052} and using this in \eqref{eqn:1035} completes the proof of the lemma. 
\end{proof}

\noindent {\bf Completing the Proof of Theorem \ref{thm:upp-root-find}(a):} Consider \eqref{eqn:945}. Using Lemma \ref{lem:no-too-small} for the first term and Lemma \ref{lem:term-i} for the second sum, we get that there exist positive constants $C_1^{\star}, C_2^{\star}, C_{11}$ independent of $M$ such that for any $K \ge K_M$ (with $D, B$ as in Lemma \ref{lem:no-too-small}), 
\begin{align}
	\limsup_{n\to\infty} \pr({v_1 \notin H_{K,\Psi}(\cT^\circ(n))}) &\leq \sum_{i=K+1}^{\infty} \frac{C_1^{\star}}{(i/2)^{C_2^{\star} M}} + 2\frac{De^{2\sqrt{(B+M)\log{(\lfloor K/2 \rfloor + 1)}}}}{(\lfloor K/2 \rfloor + 1)^{f_*/(2\overline{C}_f+\beta)}} \notag\\
	& \leq \frac{C_{11}}{(K/2)^{C_2^{\star}M-1}} + 2\frac{De^{2\sqrt{(B+M)\log{(\lfloor K/2 \rfloor + 1)}}}}{(\lfloor K/2 \rfloor + 1)^{f_*/(2\overline{C}_f+\beta)}}.  \notag
\end{align}
Choose $M = M_0$ sufficiently large such that for all $K \ge 4$,
$$
\frac{C_{11}}{(K/2)^{C_2^{\star}M_0-1}} \le \frac{D}{(\lfloor K/2 \rfloor + 1)^{f_*/(2\overline{C}_f+\beta)}}.
$$
Fixing this value of $M=M_0$, we get that there exist positive constants $C_1''$ and $C_2''$ such that 
\begin{equation}
\label{eqn:1112}
	\limsup_{n\to\infty} \pr({v_1 \notin H_{K,\Psi}(\cT^\circ(n))}) \leq 3\frac{De^{2\sqrt{(B+M_0)\log{(\lfloor K/2 \rfloor + 1)}}}}{(\lfloor K/2 \rfloor + 1)^{f_*/(2\overline{C}_f+\beta)}} \leq \frac{C_1'' e^{\sqrt{C_2'' \log{K}}}}{K^{f_*/(2\overline{C}_f + \beta)}},
\end{equation}
for all $K\geq K_{M_0}$. Choosing $K = \frac{C_1}{\eps^{(2\overline{C}_f+\beta)/f_{*}}}\exp(\sqrt{C_2\log{1/\eps}})$ (the bound in part (a) of the theorem) for appropriately chosen positive constants $C_1, C_2$ implies that 
\[\limsup_{n\to\infty}\pr({v_1 \notin H_{K,\Psi}(\cT^\circ(n))}) \leq \eps. \]
This completes the proof of part (a). 

Part (b) of the theorem follows similarly, except that now we use $\yu_{f^*}$ to bound $\BP^{(1\downarrow)}$ from above in Lemma \ref{lem:no-too-small} instead of $\LPA_{\overline{C}_f, \beta}$. We omit the details.

\subsection{Proof of Theorem \ref{thm:low-root-find}:}
We will need two lemmas to prove the theorem. Recall that $\cT_f(K)^{\circ}$ denotes the tree $\cT_f(K)$ with all labels and root information removed.  Define for $K\geq 2$, the event
\begin{equation}
\label{eqn:leaf-event-def}
	\cA_K := \set{v_1 \mbox{ is a leaf in } \cT_f(K)}.
\end{equation}
In other words, in the evolution of the process $\set{\cT_f(j):2\leq j\leq K}$, after $v_2$, all subsequent vertices attach to $v_2$ or its descendants. For $0< \gamma <1$, recall that $H_{\lfloor \gamma K \rfloor, \Psi}(\cT_f(n)^{\circ})$ denotes the $\lfloor \gamma K \rfloor$ vertices in $\cT_f(n)$ with least $\Psi$ values. 

\begin{prop}\label{prop:leaf-lower}
	Under Assumptions \ref{ass:attach-func} and \ref{ass:lim-sup}, $\exists$ $c>0$ such that $\forall \ K\geq 2$, 
	\[\liminf_{n\to\infty} \pr(v_1 \notin H_{\lfloor K/4 \rfloor, \Psi}(\cT_f(n)^{\circ})) \geq c \pr(\cA_K).\]
\end{prop}

\begin{proof}
	For $2\leq i\leq K$, let $\tilde \cT^{\sss(i)}(n) = \desc{(\cT^\circ(n),v_1)}{v_i}$.  For $i=1$, let $\tilde \cT^{\sss(1)}(n) = \desc{(\cT^\circ(n), v_2)}{v_1} = \cT_f(n) \setminus \desc{(\cT^\circ(n), v_1)}{v_2}$, i.e. we remove $v_2$ and all of it's descendants from $\cT_f(n)$ and consider the resulting tree, rooted at $v_1$. 
	 For $n \ge K$, define the two events, 
	\[\cB_{n, \lfloor K/4 \rfloor}:= \set{\Psi_{\cT^\circ(n)}(v_i) < \Psi_{\cT^\circ(n)}(v_1) \text{ for at least } \lfloor  K/4 \rfloor~ v_i's \text{ in } \set{2\leq i\leq K}}, \]
	and 
	\[\tilde{\cB}_{n, \lfloor K/4 \rfloor}:= \set{|\tilde\cT^{\sss(i)}(n)| > |\tilde\cT^{\sss(1)}(n)| \text{ for at least } \lfloor K/4 \rfloor~ v_i's \text{ in } \set{2\leq i\leq K}}.\]
	First note that for any $n \ge K$,
	\begin{align*}
	\pr(v_1 \notin H_{\lfloor K/4 \rfloor, \Psi}(\cT^\circ(n))) &\geq \pr(v_1\notin H_{\lfloor K/4 \rfloor, \Psi}(\cT^\circ(n)) \cap \cA_K )\\
	&\geq  \pr(\cB_{n, \lfloor K/4 \rfloor} \cap \cA_K) = \pr( \tilde{\cB}_{n, \lfloor K/4 \rfloor} \cap \cA_K).
	\end{align*}
	since by Lemma \ref{lem:jord-cent-det} \eqref{it:jord-nesc-suff}, $\cA_K \cap \cB_{n, \lfloor  K/4 \rfloor} = \cA_K \cap \tilde{\cB}_{n, \lfloor K/4 \rfloor}$. Define, 
\begin{align*}
\tilde\cB_{\infty, \lfloor K/4 \rfloor} &:= \liminf_{n \rightarrow \infty} \tilde{\cB}_{n, \lfloor K/4 \rfloor}\\
&= \set{|\tilde \cT^{\sss(i)}(n)| > |\tilde \cT^{\sss(1)}(n)| \text{ for at least } \lfloor K/4 \rfloor~ v_i's \text{ in } \set{2\leq i\leq K} \text { for all $n$ large}}.
\end{align*}
	Thus by Fatou's Lemma,
	\begin{equation}
	\label{eqn:541}
	\liminf_{n\to\infty} \pr(v_1 \notin H_{\lfloor K/4 \rfloor, \Psi}(\cT^\circ(n))) \geq 	\pr(\tilde\cB_{\infty, \lfloor K/4 \rfloor}|\cA_K) \pr(\cA_K).
	\end{equation}
	On the set $\cA_K$, conditional on $\cT_f(K)$  evolve the process in continuous time starting from $\cT_f(K)$. Precisely: each vertex $v\in \cT_f(K)$ with degree $\ell \geq 1$ starts a branching process with distribution $\BP_f^{\sss(\ell)}(\cdot)$ as in the setting of Lemma \ref{lem:lower-bd} where the root $v\in \cT_f(K)$ uses point process $\xi_f^{\sss(\ell)}$ for it's offspring,  while other individuals use the original point process $\xi_f$ (conditionally independent across vertices given $\cT_f(K)$). For $1 \le i \le K$, we denote by $\BP_{f,i}(\cdot)$ the branching process with root $v_i$. Thus we now have $K$ (conditionally independent) branching processes $\{\BP_{f,i}(\cdot): 1 \le i \le K\}$. 
	
	As in the setting of Lemma \ref{lem:lower-bd}, for $2\leq i\leq K$, construct $K-1$ corresponding lower bound branching processes $\{\LBP_{f,i}(\cdot): 2\leq i \leq K\}$ and let $\LW^{\sss(i)} = \liminf_{t\to\infty} e^{-\lambda t} |\LBP_{f,i}(t)|$. Note that, conditional on $\cA_K$, $\BP_{f,1}(\cdot)$ has the same distribution as the usual branching process $\BP_f(\cdot)$. Thus by Theorem \ref{prop:convg-limit}, $W_\infty^{\sss(1)} := \limsup_{t\to\infty} e^{-\lambda t} |\BP_{f,1}(t)|$ has the same distribution as $W_\infty$ in Theorem \ref{prop:convg-limit}. 
	
	For the rest of the proof, we proceed conditional on $\cA_K$. Note that $\{\LW^{\sss(i)}: 2\leq i\leq K\}$ is a collection of $K-1$ independent and identically distributed random variables, and also independent of $W_{\infty}^{\sss(1)}$. By Lemma \ref{lem:lower-bd} (b), $\LW^{\sss(2)}>0$ almost surely, and hence we can pick $w>0$ such that $\pr(\LW^{\sss(2)} > w) \ge 1/2$. For $2\leq i\leq K$, let $\hat{X}_i = \ind\set{\LW^{\sss(i)} > W_\infty^{\sss(1)}}$. Observe that
	\begin{align}\label{missing}
	\pr\left(\sum_{i=2}^K \hat{X}_i \geq K/4\right) &\ge \pr\left(\sum_{i=2}^K \hat{X}_i \geq K/4, \, W_{\infty}^{\sss(1)} \le w\right)\nonumber\\
	&\ge \pr\left(\sum_{i=2}^K \ind\set{\LW^{\sss(i)} > w} \geq K/4, \, W_{\infty}^{\sss(1)} \le w\right)\nonumber\\
	&=\pr\left(\sum_{i=2}^K \ind\set{\LW^{\sss(i)} > w} \geq K/4\right)\pr\left(W_{\infty}^{\sss(1)} \le w\right).
	\end{align}
The random variables $\{\ind\set{\LW^{\sss(i)} > w}: 2\leq i\leq K\}$ are a collection of independent Bernoulli random variables with success probability $\pr(\LW^{\sss(2)} > w) \ge 1/2$. Therefore, by Azuma-Hoeffding inequality, $\exists \ c'>0$ independent of $K$ such that  
	\[\pr\left(\sum_{i=2}^K \ind\set{\LW^{\sss(i)} > w} \geq K/4\right) \geq 1-e^{-c'K}.\]	
Moreover, since $W_{\infty}^{\sss(1)}$ has full support on $(0,\infty)$ by Theorem \ref{prop:convg-limit}, $\pr\left(W_{\infty}^{\sss(1)} \le w\right)>0$. Hence, by \eqref{missing}, we obtain $c>0$ such that $\forall \ K\geq 2$,
        \begin{align*}
		&\pr(\tilde \cB_{\infty, \lfloor K/4 \rfloor}|\cA_K)\\
		&\geq \pr(\exists~ \lfloor K/4 \rfloor ~ i's\in \set{2,\ldots, K}, \liminf_{t\to\infty} e^{-\lambda t}|\BP^{\sss(i)}(t)| > \limsup_{t\to\infty} e^{-\lambda t}|\BP^{\sss(1)}(t)||\cA_k)\\
		&= \pr(\exists~ \lfloor K/4 \rfloor ~ i's\in \set{2,\ldots, K}, \LW^{\sss(i)} > W_\infty^{\sss(1)} )\\
		&\ge \pr\left(\sum_{i=2}^K \hat{X}_i \geq K/4\right)  \ge c.
	\end{align*}
	
The proposition follows upon using the above lower bound in \eqref{eqn:541}. 
\end{proof}

\begin{prop}
	\label{prop:ak-lower}
	For $f$ satisfying Assumptions \ref{ass:attach-func} and \ref{ass:lim-sup}:  
	\begin{enumeratea}
		\item  Suppose $f(i)\geq \underline{C}_f i +\beta$ $\forall~i\geq 1$. Then $\exists \ C>0$ such that for all $K\geq 2$, $\pr(\cA_K)\geq C/K^{f(1)/(2\underline{C}_f+\beta)}$. 
		\item For general $f$ with $f_* = \inf_{i\geq 1} f(i)$,  $\exists \ C>0$ such that for all $K\geq 2$, $\pr(\cA_K) \geq C/K^{f(1)/f_*} $. 
	\end{enumeratea}
	
\end{prop}

\begin{proof}
	We will prove (a). Proof of (b) follows the exact same steps with suitable modification (described below). Recall that we start the attachment tree process $\cT_f(\cdot)$ at $n=2$ with two vertices connected by a single edge. Thus for $\cA_K$ to occur, all subsequent vertices $\{v_i : 3\leq i \leq K\}$ should {\bf not} attach to vertex $v_1$. Note that, conditioning on $\cT_f(K-1)$, we have  
	\[\pr(\cA_K) = \E\big(\ind\set{\cA_{K-1}}\pr\left(\cA_K|\cT_f(K-1)\right)\big) = \E\left[\ind\set{\cA_{K-1}} \left(1-\frac{f(1)}{\sum_{j=1}^{K-1} f(\text{deg}(v_j))}\right)\right].\]
	Note that by our assumption $\sum_{j=1}^{K-1} f(\text{deg}(v_j)) \geq \sum_{j=1}^{K-1} \left(\underline{C}_f(\text{deg}(v_j)) +\beta\right) = (2\underline{C}_f + \beta)(K-1) - \underline{C}_f$ (for part (b) one uses the lower bound $f_*(K-1)$ instead; all the subsequent steps below are identical). Let $i_0\geq 3$ be such that $\frac{f(1)}{(2\underline{C}_f+\beta)(i_0-1) - \underline{C}_f} \leq 1/2$. Assuming $K> i_0$ and using the above recursion repeatedly we get, writing $C_0 := \pr(\cA_{i_0})>0$,
	\begin{align}
	\label{eqn:ak-lb-recur}
		\pr(\cA_K) &\geq C_0\prod_{i=i_0}^{K-1}\left(1-\frac{f(1)}{(2\underline{C}_f+\beta)(i-1) - \underline{C}_f}\right)\notag\\
		&= C_0\exp\bigg(\sum_{i=i_0}^{K-1}\log{\left(1-\frac{f(1)}{(2\underline{C}_f+\beta)(i-1) - \underline{C}_f}\right)}\bigg). 
	\end{align}
	Using Taylor's expansion for $\log$, check that 
\begin{align*}
&\exp\bigg(\sum_{i=i_0}^{K-1}\log\left(1-\frac{f(1)}{(2\underline{C}_f+\beta)(i-1) - \underline{C}_f}\right)\bigg)\\
&\geq C\exp\left(-\frac{f(1)}{2\underline{C}_f+\beta}\sum_{i=i_0}^{K-1}\frac{1}{(i-1)-(\underline{C}_f/(2\underline{C}_f+\beta))}\right)
\geq C'\exp\left(-\frac{f(1)}{2\underline{C}_f+\beta} \log{K}\right),
\end{align*}
where $C>0, C'>0$ do not depend on $K$.
	Using this in \eqref{eqn:ak-lb-recur} completes the proof. 
\end{proof}

\noindent {\bf Proof of Theorem \ref{thm:low-root-find}:} We will prove part (a) of the theorem using Propositions \ref{prop:leaf-lower} and \ref{prop:ak-lower} (a). Part (b) of the theorem follows identically using Proposition \ref{prop:ak-lower} (b) instead. 

By Propositions \ref{prop:leaf-lower} and \ref{prop:ak-lower} (a)
\begin{equation}
\label{eqn:855}
	\liminf_{n\to\infty} \pr(v_1  \notin H_{\lfloor K/4 \rfloor, \Psi}(\cT_f(n)^{\circ})) \geq cC/K^{f(1)/(2\underline{C}_f+\beta)}.
\end{equation}
Thus letting $\tilde{K}(\eps) =  \lfloor \left(\frac{cC}{2\eps}\right)^{(2\underline{C}_f+\beta)/f(1)} \rfloor$ shows that for all $\eps$ small enough (to ensure $\tilde{K}(\eps) \ge 4$), 
\[\liminf_{n\to\infty} \pr(v_1 \notin H_{\lfloor \tilde{K}(\eps)/4\rfloor, \Psi}(\cT_f(n)^{\circ})) \geq 2\eps > \eps.\]
The theorem follows from this.  \qed

\section{Proofs: Persistence}
This section is dedicated to the proof of Theorem \ref{thm:persistence}. We will follow the same broad approach as in \cite{jog2016analysis,jog2018persistence} but replacing their case specific estimates by the general estimates for attachment trees derived in the previous sections.
For notational convenience, we will denote by $(v^*_{n,1}, v^*_{n,2}, \dots, v^*_{n,n})$ the vertices of the attachment tree $\cT_f(n)$ ordered in ascending order of their $\Psi$ values, with ties being broken arbitrarily. The following lemma shows that for any fixed integer $K \ge 2$, $\sup_{n \ge K}n^{-1}\Psi_{\cT_f(n)}(v^*_{n,K})$ is strictly less than one almost surely. 

\begin{lemma}\label{maxKcentral}
Suppose Assumptions \ref{ass:attach-func} and \ref{ass:lim-sup} hold. For any integer $K \ge 2$, almost surely,
$
\sup_{n \ge K}\frac{\Psi_{\cT_f(n)}(v^*_{n,K})}{n} < 1.
$
\end{lemma}

\begin{proof}
Let $v_1,\dots,v_K$ denote the first $K$ vertices in the attachment tree. For $1 \le i \le K$ and $n \ge K$, denote by  $T_{i,n}$ the subtree of $\cT_f(n)$ with $v_i$ as the root in the forest formed by removing all the edges between $v_1,\dots, v_K$. Then it follows that for any $n \ge K$ (see proof of \cite[Lemma 4.1]{jog2018persistence}),
\begin{equation}\label{p1}
\frac{\Psi_{\cT_f(n)}(v^*_{n,K})}{n} \le \max_{1 \le i \le K} \frac{\Psi_{\cT_f(n)}(v_i)}{n} \le 1 - \min_{1 \le i \le K} \frac{|T_{i,n}|}{n}.
\end{equation}
For $1 \le i \le K$, let $d_i$ denote the degree of $v_i$ in $\cT_f(K)$. By using the continuous time embedding of the dynamics, for any $1 \le i \le K$,
\begin{equation}\label{p2}
\pr\left(\inf_{n \ge K}\frac{|T_{i,n}|}{n} >0 \ \big| \ \cT_f(K)\right) = \pr\left(\inf_{t \ge 0}\frac{|\BP_f^{(d_i)}(t)|}{\sum_{j=1}^K |\BP_f^{(d_j)}(t)|} > 0 \ \big| \ \cT_f(K)\right)
\end{equation}
where, conditionally on $\cT_f(K)$, $\{\BP_f^{(d_j)}(\cdot) : 1 \le i \le K\}$ are independent branching processes defined as in the setting of Lemma \ref{lem:lower-bd}. Note that for any $\ell \in \mathbb{N}_0$ and $t \ge 0$, denoting by $\sigma_{\ell}$ the birth time of the $\ell$-th child of the root, $|\BP_f^{(\ell)}(t)|$ has the same distribution as the number of individuals in the subpopulation of $\BP_f(\sigma_{\ell} + t)$ excluding the first $\ell$ children of the root and their descendants. Thus, for any $1 \le j \le K$, as $\sigma_j< \infty$ almost surely,
\begin{align}\label{p3}
\pr\left(\sup_{t \ge 0} e^{-\lambda t} |\BP^{(d_j)}_f(t)| < \infty \ \big| \ \cT_f(K)\right) &\ge \pr\left(\sup_{t \ge 0} e^{-\lambda t} |\BP_f(t + \sigma_j)| < \infty\right)\notag\\
&= \pr\left(\limsup_{t \rightarrow \infty} e^{-\lambda t} |\BP_f(t)| < \infty\right)\notag\\
&=\pr\left(W_{\infty} < \infty \right) = 1,
\end{align}
where the first equality is due to the fact that $\BP_f(\cdot)$ is non-explosive (which follows from Assumption \ref{ass:attach-func} (ii)), and the second equality is due to Theorem \ref{prop:convg-limit}. Recall the lower bounding process $\LBP_f$ in Lemma \ref{lem:lower-bd}. By Lemma \ref{lem:lower-bd} (b),
\begin{equation}\label{p4}
\pr\left(\inf_{t \ge 0} e^{-\lambda t} |\BP^{(d_i)}_f(t)| > 0 \ \big| \ \cT_f(K)\right) \ge \pr\left(\inf_{t \ge 0} e^{-\lambda t} |\LBP_f(t)| > 0\right) = \pr(\underline{W}_{\infty} >0) = 1,
\end{equation}
where the second-to-last equality holds because $|\LBP_f(t)|>0$ for all $t \ge 0$. 
Thus, using \eqref{p3} and \eqref{p4} in \eqref{p2}, for any $1 \le i \le K$,
\begin{multline*}
\pr\left(\inf_{n \ge K}\frac{|T_{i,n}|}{n} >0 \ \big| \ \cT_f(K)\right) = \pr\left(\inf_{t \ge 0}\frac{e^{-\lambda t}|\BP_f^{(d_i)}(t)|}{\sum_{j=1}^K e^{-\lambda t}|\BP_f^{(d_j)}(t)|} > 0 \ \big| \ \cT_f(K)\right)\\
\ge \pr\left(\inf_{t \ge 0} e^{-\lambda t} |\BP^{(d_i)}_f(t)| > 0 \text{ and } \sup_{t \ge 0} e^{-\lambda t} |\BP^{(d_j)}_f(t)| < \infty \text{ for all } 1 \le j \le K \ \big| \ \cT_f(K)\right) = 1.
\end{multline*}
The lemma follows from this and \eqref{p1}.
\end{proof}
Fix any integer $K \ge 2$. For $n \ge K$, let $\mathcal{K}_n$ denote the set of all vertices $v$ in $\cT_f(n)$ with $\Psi_{\cT_f(n)}(v) \le \Psi_{\cT_f(n)}(v^*_{n,K})$. Note that this set can have more than $K$ elements if there are ties in $\Psi$ values. Define $\mathcal{K} := \cup_{n \ge K} \mathcal{K}_n$. This set consists of the vertices that take one of the $K$ smallest $\Psi$ values at some point of the tree evolution. We will show the following.
\begin{lemma}\label{finKset}
Under Assumptions \ref{ass:attach-func} and \ref{ass:lim-sup}, $|\mathcal{K}| < \infty$  almost surely.
\end{lemma}

\begin{proof}
For $\zeta \in (0,1)$, define the event
$$
B_{\zeta} := \{\Psi_{\cT_f(n)}(v^*_{n,K}) \le \zeta n \text{ for all } n \ge K\}.
$$
For $n \ge K$, define
$$
\mathcal{H}_n := \{v_n \in \mathcal{K}\},
$$
namely, the event that the $n$-th vertex in the attachment tree has one of the $K$ least $\Psi$ values at some point in the future evolution of the tree. 
Fix any $\zeta \in (0,1)$. 
By Lemma \ref{lem:jord-cent-det} (a), for any $n \ge K$ and any $m \ge n$,
$$
\Psi_{\cT_f(m)}(v_n) \le \Psi_{\cT_f(m)}(v^*_{n,K}) \Longleftrightarrow |\desc{(\cT^\circ(m), v^*_{n,K})}{v_n}|\geq |\desc{(\cT^\circ(m), v_n)}{v^*_{n,K}}|.
$$
Let $\{v^*_{n,K}, u_1,\dots,v_n\}$ denote the path in $\cT_f(n)$ between $v^*_{n,K}$ and $v_n$. Then, for any $n \ge K$, if the event $\{\Psi_{\cT_f(n)}(v^*_{n,K}) \le \zeta n\}$ holds, we have
\begin{equation}\label{ps3}
|\desc{(\cT^\circ(n), v_n)}{v^*_{n,K}}| = n - |\desc{(\cT^\circ(n), v^*_{n,K})}{u_1}| \ge n - \Psi_{\cT_f(n)}(v^*_{n,K}) \ge n(1-\zeta).
\end{equation}
Now, conditional on $\cT_f(n)$, we run the attachment process in continuous time starting from initial configurations $\desc{(\cT^\circ(n), v_n)}{v^*_{n,K}}$ and $\desc{(\cT^\circ(n), v^*_{n,K})}{v_n}$ as in the proof of Lemma \ref{lem:term-i} where each vertex in one of these two initial population graphs reproduces according to the point process shifted by its degree in the corresponding graph and the descendants of these vertices reproduce according to the original point process. For $t \ge 0$, write $\mathcal{Z}^{(1)}(t)$ and $\mathcal{Z}^{(2)}(t)$ for the total number of descendants of vertices in $\desc{(\cT^\circ(n), v_n)}{v^*_{n,K}}$ and $\desc{(\cT^\circ(n), v^*_{n,K})}{v_n}$ respectively (counting the vertices in the initial configurations). Then, $\mathcal{Z}^{(2)}(\cdot)$ has the same distribution as $Z_f(\cdot)$. Further, note that if the event $\{\Psi_{\cT_f(n)}(v^*_{n,K}) \le \zeta n\}$ holds, we have by Lemma \ref{lem:lower-bd} and \eqref{ps3},
$$
\mathcal{Z}^{(1)}(\cdot) \ge \sum_{j=1}^{\lfloor (1-\zeta)n \rfloor}\LZF^{(j)}(\cdot),
$$
where $\{\LZF^{(j)}(\cdot) : j \ge 1\}$ are \emph{i.i.d.} copies of the lower bounding process $\LZF(\cdot)$ in Lemma \ref{lem:lower-bd}. Thus, writing $\underline{\mu} := \mathbb{E}\left(\inf_{t \ge 0} e^{-\lambda t}\LZF(t)\right)$ (which is positive by Lemma \ref{lem:lower-bd} (b)), for any $n \ge K$,
\begin{align}\label{ps4}
&\pr\left(B_{\zeta} \cap \mathcal{H}_n\right)\\
&\le \pr\left(\{\Psi_{\cT_f(n)}(v^*_{n,K}) \le \zeta n\} \cap \{|\desc{(\cT^\circ(m), v^*_{n,K})}{v_n}|\geq |\desc{(\cT^\circ(m), v_n)}{v^*_{n,K}}| \text{ for some } m \ge n\}\right)\notag\\
&\le \pr\left(\{\Psi_{\cT_f(n)}(v^*_{n,K}) \le \zeta n\} \cap \{\mathcal{Z}^{(2)}(t) \ge \mathcal{Z}^{(1)}(t) \text{ for some } t \ge 0\} \ \big| \ \cT_f(n)\right)\notag\\
&\le \pr\left(\mathcal{Z}^{(2)}(t) \ge \sum_{j=1}^{\lfloor(1-\zeta)n\rfloor}\LZF^{(j)}(t) \text{ for some } t \ge 0\right)\notag\\
&\le \pr\left(\sup_{t \ge 0}e^{-\lambda t}\mathcal{Z}^{(2)}(t) \ge n(1-\zeta) \underline{\mu}/2\right) + \pr\left( \sum_{j=1}^{\lfloor(1-\zeta)n\rfloor}\inf_{t \ge 0}(e^{-\lambda t}\LZF^{(j)}(t)) \le n(1-\zeta) \underline{\mu}/2\right).\notag
\end{align}
By Corollary \ref{cor:sup-bound}, there exists $C_1 >0$ such that for any $n \ge K$,
\begin{equation}\label{ps5}
\pr\left(\sup_{t \ge 0}e^{-\lambda t}\mathcal{Z}^{(2)}(t) \ge n(1-\zeta) \underline{\mu}/2\right) \le \frac{C_1 \log (n(1-\zeta) \underline{\mu})}{(n(1-\zeta) \underline{\mu})^2}.
\end{equation}
Note that for any $j \ge 1$, $\inf_{t \ge 0}(e^{-\lambda t}\LZF^{(j)}(t)) \le 1$. Thus, by Azuma-Hoeffding inequality, there exists $C_2 >0$ such that for any $n \ge K$,
\begin{multline}\label{ps6}
\pr\left( \sum_{j=1}^{\lfloor(1-\zeta)n\rfloor}\inf_{t \ge 0}(e^{-\lambda t}\LZF^{(j)}(t)) \le n(1-\zeta) \underline{\mu}/2\right)\\
\le \pr\left( \sum_{j=1}^{\lfloor(1-\zeta)n\rfloor}\left(\inf_{t \ge 0}(e^{-\lambda t}\LZF^{(j)}(t)) - \underline{\mu}\right) \le 1-n(1-\zeta) \underline{\mu}/2\right) \le e^{-C_2(1-\zeta)n}.
\end{multline}
Using \eqref{ps5} and \eqref{ps6} in \eqref{ps4}, we conclude that $\sum_{n \ge K}\pr\left(B_{\zeta} \cap \mathcal{H}_n\right) < \infty$. Thus, for any $\zeta \in (0,1)$, $|\mathcal{K}| < \infty$ almost surely on the event $B_{\zeta}$. The lemma now follows upon noting from Lemma \ref{maxKcentral} that $\pr\left(B_{\zeta}\right) \rightarrow 1$ as $\zeta \rightarrow 1$.
\end{proof}

\noindent {\bf Proof of Theorem \ref{thm:persistence}:} For any $1 \le i < j$, it follows from Theorem \ref{prop:convg-limit} that 
$$
\frac{|\desc{(\cT^\circ(n), v_i)}{v_j}|}{|\desc{(\cT^\circ(n), v_j)}{v_i}|} \stackrel{a.s.}{\longrightarrow} \frac{W_{\infty}^j}{W_{\infty}^i}, \qquad \mbox{ as } n\to\infty
$$
for absolutely continuous or singular continuous independent random variables whose distributions are determined by the tree $\cT_f(j)$. In either case, $\pr(W_{\infty}^j = W_{\infty}^i) = 0$ which, along with Lemma \ref{lem:jord-cent-det} (a), gives a random $n_{i,j} \in \mathbb{N}$ such that $\Psi_{\cT_f(n)}(v_i) \neq \Psi_{\cT_f(n)}(v_j)$ for all $n \ge n_{i,j}$. Hence, for any $N \in \mathbb{N}$, almost surely, there exists a (random) permutation $(\pi(1),\dots, \pi(N))$ of $\{1,\dots,N\}$ and $n_N \in \mathbb{N}$ such that
$$
\Psi_{\cT_f(n)}(v_{\pi(1)}) < \cdots < \Psi_{\cT_f(n)}(v_{\pi(N)}) \ \text{ for all } n \ge n_N.
$$
This observation readily implies for any $K, N \in \mathbb{N}$ with $N \ge K$, 
$$
\pr(\{\cT_f(n) : n \ge 1\} \text{ is } (\Psi,K) \text{ persistent}, \  \mathcal{K} \subset \{1,\dots, N\}) = \pr(\mathcal{K} \subset \{1,\dots, N\}).
$$
The theorem now follows by taking a limit as $N \rightarrow \infty$ on both sides of the equality above and using Lemma \ref{finKset}.

\section*{Acknowledgements}
Bhamidi was partially supported by NSF grants DMS-1613072, DMS-1606839 and ARO grant W911NF-17-1-0010. We acknowledge valuable feedback from two anonymous referees which led to major improvements in the presentation and clarity of this article.

\bibliographystyle{imsart-number}
\bibliography{persistence,pref_change_bib,scaling}
\end{document}